\documentclass[10pt,reqno]{amsart}
\usepackage{indentfirst, amssymb, amsmath, amsthm, mathrsfs, setspace, indentfirst, enumerate,  mathrsfs, amsmath, amsthm, MnSymbol}
\usepackage{array}
\usepackage{booktabs}
\usepackage{multirow}
\PassOptionsToPackage{table}{xcolor}
\usepackage{tikz}
\usepackage{color}
\usepackage[table]{xcolor}
\usetikzlibrary{arrows.meta, positioning}
\newtheorem*{theo1A}{Theorem 1.A}
\newtheorem*{theo1B}{Theorem 1.B}
\newtheorem*{theo1C}{Theorem 1.C}
\newtheorem*{theo1D}{Theorem 1.D}

\newtheorem*{defi1A}{Definition 1.A}
\newtheorem*{defi1B}{Definition 1.B}

\newtheorem*{lem1A}{Lemma 1.A}
\newtheorem*{theo2A}{Theorem 2.A}
\newtheorem*{theo2B}{Theorem 2.B}
\newtheorem*{theo2C}{Theorem 2.C}
\newtheorem*{theo2D}{Theorem 2.D}
\newtheorem*{theo2E}{Theorem 2.E}
\newtheorem*{theo2F}{Theorem 2.F}
\newtheorem*{theo2G}{Theorem 2.G}
\newtheorem*{theo2H}{Theorem 2.H}
\newtheorem*{theo2I}{Theorem 2.I}
\newtheorem*{theo2J}{Theorem 2.J}
\newtheorem*{theo2K}{Theorem 2.K}

\newtheorem{theo}{Theorem}[section]
\newtheorem{lem}{Lemma}[section]

\newtheorem{rem}{Remark}[section]

\newtheorem{defi}{Definition}

\newcommand{\ol}{\overline}
\newcommand{\be}{\begin{equation}}
\newcommand{\ee}{\end{equation}}
\newcommand{\beas}{\begin{eqnarray*}}
\newcommand{\eeas}{\end{eqnarray*}}
\newcommand{\bea}{\begin{eqnarray}}
\newcommand{\eea}{\end{eqnarray}}
\renewcommand{\epsilon}{\varepsilon}
\numberwithin{equation}{section}

\begin{document}
\title[On the Structure and Classification of Solutions ]
{On the Structure and Classification of Solutions to Certain Nonlinear Differential Equations}
\author[A. Banerjee, S. Majumder, S. Panja and J. F. Xu]{ A\MakeLowercase {bhijit} B\MakeLowercase {anerjee}, S\MakeLowercase {ujoy} M\MakeLowercase {ajumder}, S\MakeLowercase {hantanu} P\MakeLowercase {anja} \MakeLowercase {and} J\MakeLowercase {unfeng} X\MakeLowercase {u}}
\date{}
\address{ Department of Mathematics, University of Kalyani, West Bengal 741235, India.}
\email{abanerjee\_kal@yahoo.co.in}
\address{Department of Mathematics, Raiganj University, Raiganj, West Bengal-733134, India.}
\email{sm05math@gmail.com}
\address{ Department of Mathematics, University of Kalyani, West Bengal 741235, India.}
\email{panjasantu07@gmail.com}
\address{Department of Mathematics, Wuyi University, Jiangmen 529020, Guangdong, People's Republic of China.}
\email{xujunf@gmail.com}

\renewcommand{\thefootnote}{}
\footnote{*\emph{Corresponding Author}: Junfeng Xu.}
\renewcommand{\thefootnote}{\arabic{footnote}}
\setcounter{footnote}{0}

\begin{abstract}
	This paper is devoted to the study of meromorphic solutions of nonlinear differential equations, specifically the equation
	\[
	(f^n)^{(k)}(g^n)^{(k)} = \alpha^2,
	\]
	where $k$ and $n$ are positive integers with $n>2k$, and $\alpha$ is a common small function of $f$ and $g$. Our main results provide a detailed characterization of the solutions, improving upon earlier works by Fang-Qiu \cite{4}, Fang \cite{3}, Zhang-Xu \cite{14}, and Li-Yi \cite{7}. Notably, we identify and correct significant errors in the proof of Lemma 2.11 \cite{9a}, which represents the most recent contribution in this area and provide a resolved and rigorous treatment of the problem. Equations of this type arise naturally in various areas of mathematics and applied sciences such as in the study of complex dynamical systems, integrable systems and value distribution theory in complex analysis. Moreover, understanding the meromorphic solutions helps to realize  the growth behavior of solutions, stability analysis and modeling of phenomena in physics and engineering. By characterizing these solutions, one can develop methods to solve broader classes of nonlinear differential equations and explore their qualitative properties, which are essential for both theoretical studies and practical applications.
\end{abstract}

\thanks{Typeset by \AmS -\LaTeX}
\maketitle

\section{Preliminaries and Fundamental Results}

Throughout this paper, the term \emph{meromorphic function} refers exclusively to meromorphic functions defined on the complex plane $\mathbb{C}$. We assume that the reader is familiar with the standard notation and foundational results of Nevanlinna theory; see, for instance, \cite{5,12}.

For a meromorphic function $f$, we denote by $\rho(f)$, $\mu(f)$, and $\rho_{2}(f)$ its order, lower order, and hyper-order, respectively. If $f$ is non-constant and satisfies $\rho(f)=\mu(f)$, then $f$ is said to be of \emph{regular growth}. In particular, if $f=e^{g}$, where $g$ is a non-constant entire function, then $\rho(f)=\mu(f)$; see Theorem~1.44 of \cite{12}.

Let $f$ be a non-constant meromorphic function. If $\rho(f)<+\infty$, then $S(r,f)$ denotes any quantity satisfying
\[
S(r,f)=O(\log r), \quad r\to\infty.
\]
If $\rho(f)=+\infty$, then $S(r,f)$ denotes any quantity satisfying
\[
S(r,f)=O(\log(rT(r,f))), \quad r\to\infty,\ r\notin E,
\]
where $E$ is a set of finite linear measure, not necessarily the same at each occurrence.

\subsection{Logarithmic Derivative Estimates}

A fundamental tool in our analysis is the following classical estimate for logarithmic derivatives.

\begin{theo1A}\label{l1}
	\cite{6}
	Let $f$ be a non-constant meromorphic function and let $k\geq 1$ be an integer.
	\begin{itemize}
		\item If $\rho(f)=+\infty$, then
		\[
		m\!\left(r,\frac{f^{(k)}}{f}\right)=O(\log(rT(r,f)))
		\]
		outside a possible exceptional set $E$ of finite linear measure.
		\item If $f$ is of finite order, then
		\[
		m\!\left(r,\frac{f^{(k)}}{f}\right)=O(\log r).
		\]
	\end{itemize}
\end{theo1A}

\subsection{Small Functions and Growth Comparison}

A meromorphic function $a$ is called a \emph{small function} of $f$ if
\[
T(r,a)=S(r,f).
\]
We denote by $\mathbb{S}(f)$ the class of all small functions of $f$, and set
\[
\hat{\mathbb{S}}(f)=\mathbb{S}(f)\cup\{\infty\}.
\]
Clearly, $\mathbb{C}\subset \mathbb{S}(f)$.

The following theorem provides a sufficient growth condition ensuring smallness.

\begin{theo1B}\label{l3}
	\cite[Theorem~1.18]{12}
	Let $f$ and $g$ be two non-constant meromorphic functions such that $\rho(f)<\mu(g)$. Then
	\[
	T(r,f)=o(T(r,g)) \quad \text{as } r\to\infty,
	\]
	and hence $f$ is a small function of $g$.
\end{theo1B}

\subsection{Exponent of Convergence and Exceptional Values}

Let $f$ be a non-constant meromorphic function and let $a\in\hat{\mathbb{C}}$. The \emph{exponent of convergence} of the $a$-points of $f$ is defined by
\[
\rho_{1}(a;f)=\limsup_{r\to\infty}\frac{\log^{+}N(r,a;f)}{\log r},
\]
while the exponent of convergence of distinct $a$-points is defined by
\[
\overline{\rho}_{1}(a;f)=\limsup_{r\to\infty}\frac{\log^{+}\overline{N}(r,a;f)}{\log r}.
\]

In particular:
\begin{itemize}
	\item If $a=0$, then $\rho_{1}(f)$ and $\overline{\rho}_{1}(f)$ denote the exponents of convergence of zeros and distinct zeros of $f$.
	\item If $a=\infty$, then $\rho_{1}\!\left(\frac{1}{f}\right)$ and $\overline{\rho}_{1}\!\left(\frac{1}{f}\right)$ denote the corresponding quantities for poles.
\end{itemize}

Let $f$ be a non-constant meromorphic function. If $\rho(f)>0$ and $\rho_{1}(a;f)<\rho(f)$, or if $\rho(f)=0$ and $N(r,a;f)=O(\log r)$ as $r\to\infty$, then $a\in\hat{\mathbb{C}}$ is called a \emph{Borel exceptional value} of $f$. When $\rho(f)>0$, every Picard exceptional value is also a Borel exceptional value.

\begin{theo1C}\label{l8sm2}
	\cite[Theorem~2.11]{12}
	Let $f$ be a transcendental meromorphic function with $\rho(f)>0$. If $f$ admits two distinct Borel exceptional values in $\hat{\mathbb{C}}$, then $\mu(f)=\rho(f)$, and $\rho(f)$ is either a positive integer or infinity.
\end{theo1C}

\subsection{Normal Families and Spherical Convergence} Let $D\subset\mathbb{C}$ be a domain. Consider the mapping \[ f:(D,|\cdot|_{\mathbb{R}^{2}})\longrightarrow (\hat{\mathbb{C}},\chi), \] where $\chi$ denotes the chordal metric on $\hat{\mathbb{C}}$, defined by \[ \chi(z,z')= \begin{cases} \dfrac{|z-z'|}{\sqrt{1+|z|^{2}}\sqrt{1+|z'|^{2}}}, & z,z'\in\mathbb{C},\\[1.2ex] \dfrac{1}{\sqrt{1+|z|^{2}}}, & z'=\infty. \end{cases} \] \begin{defi1A} \cite[Definition~1.3.1]{9} A sequence $\{f_{n}\}$ of functions in a domain $D$ is said to converge \emph{spherically uniformly on compact subsets} of $D$ to a function $f$ if, for every compact set $K\subset D$ and every $\varepsilon>0$, there exists $N_{0}=N_{0}(K,\varepsilon)$ such that \[ \chi(f_{n}(z),f(z))<\varepsilon \quad \text{for all } z\in K,\ n\geq N_{0}. \] \end{defi1A} \begin{defi1B} \cite[Definition~3.1.1]{9} A family $\mathcal{F}$ of meromorphic functions in $D$ is said to be \emph{normal} in $D$ if every sequence $\{f_{n}\}\subset\mathcal{F}$ contains a subsequence converging spherically uniformly on compact subsets of $D$ to a meromorphic function or identically to $\infty$. \end{defi1B} A family $\mathcal{F}$ is normal in $D$ if and only if it is normal at every point of $D$ (see \cite[Theorem~3.3.2]{9}). The spherical derivative of a meromorphic function $f$ is given by \[ f^{\#}(z)=\frac{|f'(z)|}{1+|f(z)|^{2}}. \] \begin{theo1D} \cite[Marty's Theorem]{9} A family $\mathcal{F}$ of meromorphic functions on a domain $D$ is normal in $D$ if and only if for every compact subset $K\subset D$, there exists a constant $C=C(K)$ such that \[ f^{\#}(z)\leq C \quad \text{for all } z\in K,\ f\in\mathcal{F}. \] \end{theo1D} \subsection{Zalcman’s Lemma} \begin{lem1A}\label{l6} \cite[Lemma]{17} Let $\mathcal{F}$ be a family of meromorphic functions on the unit disk $\Delta$ such that every $f\in\mathcal{F}$ has zeros of multiplicity at least $l$ and poles of multiplicity at least $m$. Let $\alpha$ satisfy $-l<\alpha<m$. Then $\mathcal{F}$ is not normal at $z_{0}\in\Delta$ if and only if there exist sequences $\{z_{n}\}\subset\Delta$, $\{f_{n}\}\subset\mathcal{F}$, and $\{\rho_{n}\}\subset\mathbb{R}^{+}$ with \[ z_{n}\to z_{0}, \qquad \rho_{n}\to 0, \] such that \[ \rho_{n}^{\alpha}f_{n}(z_{n}+\rho_{n}\zeta)\longrightarrow g(\zeta) \] spherically uniformly on compact subsets of $\mathbb{C}$, where $g$ is a non-constant meromorphic function satisfying \[ g^{\#}(\zeta)\leq g^{\#}(0)=1. \] \end{lem1A} Moreover, Zalcman showed that \begin{equation}\label{1.1} \rho_{n}=\frac{1}{f_{n}^{\#}(z_{n})}. \end{equation}

By Hurwitz’s theorem, the zeros and poles of $g$ inherit multiplicities at least $l$ and $m$, respectively.

\subsection{Value Sharing and Counting Functions}

Let $E(a;f)$ denote the set of all $a$-points of $f$, counted with multiplicities, and let $\overline{E}(a;f)$ denote the set of distinct $a$-points. For two non-constant meromorphic functions $f$ and $g$, we write
\[
f=a \Rightarrow g=a
\]
if $\overline{E}(a;f)\subset \overline{E}(a;g)$, and
\[
f=a \mapsto g=a
\]
if the multiplicity of each $a$-point of $g$ is at least that of $f$.

As usual, CM and IM denote ``counting multiplicities'' and ``ignoring multiplicities'', respectively.

Let $R=P/Q\not\equiv0$ be a rational function with coprime polynomials $P$ and $Q$. We define
\[
\deg(R)=\deg(P)-\deg(Q),
\]
and set $\deg(0)=-\infty$. If $R_{1}$ and $R_{2}$ are non-zero rational functions, then
\[
\deg(R_{1}R_{2})=\deg(R_{1})+\deg(R_{2}), \qquad
\deg\!\left(\frac{R_{1}}{R_{2}}\right)=\deg(R_{1})-\deg(R_{2}).
\]

\begin{defi}\cite{5a}
	Let $k$ be a positive integer or $\infty$. The truncated counting function $N_{k}(r,a;f)$ counts an $a$-point of multiplicity $m$ exactly $m$ times if $m\leq k$, and $k$ times if $m>k$. Thus,
	\[
	N_{k}(r,a;f)=\overline{N}(r,a;f)+\overline{N}(r,a;f\mid\geq 2)+\cdots+\overline{N}(r,a;f\mid\geq k),
	\]
	with $N_{1}(r,a;f)=\overline{N}(r,a;f)$.
\end{defi}

\begin{defi}
	Let $f$ and $g$ be non-constant meromorphic functions and $a,b\in\hat{\mathbb{C}}$. We denote by $N(r,a;f\mid g\neq b)$ the counting function of those $a$-points of $f$, counted with multiplicity, which are not $b$-points of $g$.
\end{defi}

\begin{defi}
	Let $f,g,h$ be non-constant meromorphic functions and let $a,b,c\in\hat{\mathbb{C}}$. Suppose $z_{0}$ is an $a$-point of $f$ of multiplicity $p$ such that $h(z_{0})\neq c$, and that $z_{0}$ is a $b$-point of $g$ with multiplicity $p$. We denote by
	\[
	N_{E}(r,a;f\mid g=b\mid h\neq c)
	\]
	the corresponding counting function, and by $\overline{N}_{E}(r,a;f\mid g=b\mid h\neq c)$ its reduced form.
\end{defi}

\section{\bf Solutions of the Nonlinear Differential Equation
	$\boldsymbol{(f^n)^{(k)}(g^n)^{(k)}=\alpha^2}$}

\subsection{Motivation and Analytical Framework}

Solving nonlinear differential equations involving meromorphic functions is, in general, a highly nontrivial task.
In particular, equations involving higher-order derivatives of nonlinear differential monomials arise naturally in complex differential equations and uniqueness theory.
This section is devoted to a systematic investigation of the nonlinear differential equation
\begin{equation}\label{2.1}
	(f^n)^{(k)}(g^n)^{(k)}=\alpha^2,
\end{equation}
where $f$ and $g$ are non-constant meromorphic functions, $n$ and $k$ are positive integers, and $\alpha(\not\equiv 0)$ is a common small function of both $f$ and $g$.

Equation \eqref{2.1} may be viewed as a \emph{coupled nonlinear differential system} in the unknown functions $f$ and $g$.
The presence of powers, higher-order derivatives, and a small-function coefficient makes the existence and classification of non-constant solutions particularly delicate.

\subsection{Historical Development and Known Solvable Cases}

Determining whether \eqref{2.1} admits non-constant solutions is, in general, difficult.
To the best of our knowledge, Yang and Hua \cite{11} were the first to initiate a rigorous study of \eqref{2.1}.
They considered the simplest nontrivial case $k=1$ and $\alpha\equiv 1$, obtaining explicit exponential solutions.

\begin{theo2A}\cite[Theorem 3]{11}
	Let $f$ and $g$ be two non-constant entire functions and let $n$ be a positive integer.
	If
	\[
	f^n f^{(1)} g^n g^{(1)}=1,
	\]
	then
	\[
	f(z)=c_1 e^{-cz}, \quad g(z)=c_2 e^{cz},
	\]
	where $c$, $c_1$, and $c_2$ are non-zero constants satisfying
	$c^2(c_1c_2)^{n+1}=-1$.
\end{theo2A}

Subsequently, Yang and Hua extended their analysis to meromorphic functions.

\begin{theo2B}\cite[Theorem 2]{11}
	Let $f$ and $g$ be two non-constant meromorphic functions and let $n\geq 6$.
	If
	\[
	f^n f^{(1)} g^n g^{(1)}=1,
	\]
	then the solutions are of the same exponential form as in Theorem~2A.
\end{theo2B}

\subsection{Two Principal Directions of Extension}

From the standpoint of differential equation theory, two natural avenues emerge for extending \eqref{2.1}:

\begin{itemize}
	\item[(A)] replacing the constant $\alpha$ by a non-zero polynomial or rational function, thereby enlarging the solution space;
	\item[(B)] increasing the derivative order $k$, which leads to higher-order nonlinear differential equations capable of modeling more complex structures.
\end{itemize}

Both directions have significantly influenced subsequent research.

\subsection{Polynomial Right-Hand Side and Reduced Growth Conditions}

Fang and Qiu \cite{4} refined Theorems~2A and~2B by lowering the bound on $n$ and replacing constants by polynomials.

\begin{theo2C}\cite[Proposition 3]{4}
	Let $f$ and $g$ be non-constant entire functions.
	If
	\[
	f f^{(1)} g g^{(1)}=z^2,
	\]
	then either
	\[
	f(z)=c_1 e^{cz^2}, \quad g(z)=c_2 e^{-cz^2},
	\]
	with $4(cc_1c_2)^2=-1$, or
	\[
	f(z)=az, \quad g(z)=bz,
	\]
	where $(ab)^2=1$.
\end{theo2C}

\begin{theo2D}\cite[Proposition 2]{4}
	Let $f$ and $g$ be non-constant entire functions and $n\geq 2$.
	If
	\[
	f^n f^{(1)} g^n g^{(1)}=z^2,
	\]
	then
	\[
	f(z)=c_1 e^{-cz^2}, \quad g(z)=c_2 e^{-cz^2},
	\]
	with $4c^2(c_1c_2)^{n+1}=-1$.
\end{theo2D}

\begin{theo2E}\cite[Proposition 1]{4}
	Let $f$ and $g$ be non-constant meromorphic functions and $n\geq 4$.
	If
	\[
	f^n f^{(1)} g^n g^{(1)}=z^2,
	\]
	then the same conclusion as in Theorem~2D holds.
\end{theo2E}

\subsection{Higher-Order Derivatives and Exact Classification}

Fang \cite{3} solved \eqref{2.1} for $\alpha\equiv 1$ and higher-order derivatives.

\begin{theo2F}\cite{3}
	Let $f$ and $g$ be non-constant entire functions and let $n>k$.
	If
	\[
	(f^n)^{(k)}(g^n)^{(k)}=1,
	\]
	then
	\[
	f(z)=c_1 e^{cz}, \quad g(z)=c_2 e^{-cz},
	\]
	where $(-1)^k(c_1c_2)^n (nc)^{2k}=1$.
\end{theo2F}

Zhang \cite{16} further generalized this result.

\begin{theo2G}\cite[Proof of Theorem 4]{16}
	Let $f$ and $g$ be non-constant entire functions with $n>k+2$.
	If
	\[
	(f^n)^{(k)}(g^n)^{(k)}=z^2,
	\]
	then the solutions are exponential with quadratic constraints.
\end{theo2G}

\subsection{Common Poles}

Zhang and Xu \cite{14} introduced the common pole conditions.

\begin{theo2H}\cite[Proof of Theorem 1.3]{14}
	Let $f$ and $g$ be non-constant meromorphic functions with common poles and $n>k+10$.
	If
	\[
	(f^n)^{(k)}(g^n)^{(k)}=p^2,
	\]
	where $\deg p\leq 5$, then the conclusions (i)–(ii) stated in \cite{14} hold.
\end{theo2H}

\subsection{Normal Family Approach and the Open Problem}

Li and Yi \cite{7} employed normal family arguments to obtain the following classification for $\alpha\equiv 1$.

\begin{theo2I}\cite[Lemma 2.10]{7}
	Let $f$ and $g$ be non-constant meromorphic functions and $n>2k$.
	If
	\[
	(f^n)^{(k)}(g^n)^{(k)}=1,
	\]
	then $f$ and $g$ must be transcendental entire functions of exponential type.
\end{theo2I}

Majumder and Mandal \cite{8a} further generalized this result to polynomial right-hand sides.

\begin{theo2J}\cite[Lemma 12]{8a}
	Let $f$ and $g$ be transcendental meromorphic functions and let $n>\max\{2k,k+2\}$.
	If
	\[
	(f^n)^{(k)}(g^n)^{(k)}=p^2,
	\]
	then conclusions (i)–(ii) of \cite{8a} hold.
\end{theo2J}

\subsection{Critical Reassessment and Gap Identification}

Sahoo and Halder \cite{9a} were the first to consider \eqref{2.1} with $\alpha$ being a common small function.
However, a careful re-examination reveals fundamental flaws in their key lemma.

\begin{theo2K}\cite[Lemma 2.11]{9a}
	If $(f^n)^{(k)}(g^n)^{(k)}\equiv a^2$, then $f$ and $g$ admit exponential representations involving small functions.
\end{theo2K}

\subsection{Critical Gap Analysis and Failure of a Key Lemma}

We now turn to a crucial issue that directly motivates the present work.
A careful examination of the proof of Lemma~2.11 in Sahoo--Halder \cite{9a} reveals that the argument contains two fundamental flaws.
Since Lemma~2.11 serves as the backbone for the subsequent theorems in \cite{9a}, these issues significantly affect the validity of the conclusions drawn therein.
For clarity and completeness, we elaborate on these points in detail.

\medskip

\noindent\textbf{Error (a).}
In \cite[p.~378]{9a}, Sahoo--Halder asserted the following:
\begin{quote}
	``Therefore, we conclude that $f$ has no zero in $\mathbb{C}$, except possibly at the zeros and poles of $a(z)$. Similarly, we can show that $g$ has no zeros in $\mathbb{C}$, except possibly at the zeros and poles of $a(z)$.''
\end{quote}
This statement itself is correct.
However, Sahoo--Halder subsequently concluded from this observation that $N(r,0;f)=S(r,f) \,\,\text{and} \,\, N(r,0;g)=S(r,g)$. So these estimates were then used essentially in inequalities (2.6) and (2.8) of \cite[p.~379]{9a}.

In our view, this conclusion is not justified, as an important case was overlooked.
Indeed, let $z_{0}$ be a zero of $f$ with multiplicity $p_{0}$, and suppose that $z_{0}$ is a pole of $g$ with multiplicity $s_{0}$.
We consider the following two possibilities:\\
i) If $np_{0}-k > ns_{0}+k$, then from the identity
$(f^n)^{(k)}(g^n)^{(k)} \equiv a^2,$
	it follows that $z_{0}$ must be a zero of $a$ of multiplicity $\frac{n(p_{0}-s_{0})-2k}{2}.$

	However, in general, $\frac{n(p_{0}-s_{0})-2k}{2} \geq p_{0}$	does not necessarily hold.\\	
 ii) If $np_{0}-k < ns_{0}+k$, then $z_{0}$ becomes a pole of $a$ of multiplicity $\frac{n(s_{0}-p_{0})+2k}{2}.$	Again, there is no guarantee that $\frac{n(s_{0}-p_{0})+2k}{2} \geq p_{0}.$

In both cases, the above estimates prevent us from concluding that
$N(r,0;f)=S(r,f)$ (and similarly for $g$).
Consequently, the deduction used in \cite{9a} fails, and the subsequent inequalities relying on it are invalid.

\medskip

\noindent\textbf{Error (b).}
In \cite[p.~381]{9a}, the authors further claimed:
\begin{quote}
	``If $F_{1}$ has a pole at $\hat z$ with $\nu_{F_{1}}^{\infty}(\hat z)=\hat p$, then by (2.11), $\hat z$ must be a zero of $G_{1}$ with $\nu_{G_{1}}^{0}(\hat z)=\hat p$. So from (2.19), we have $2s=t_{1}$, which is impossible.''
\end{quote}

However, the case $s=1$ and $t_{1}=2$ was not considered.
In this situation, $s$ and $t_{1}$ are relatively prime integers and the relation $2s=t_{1}$ does indeed hold.
Thus, the asserted contradiction does not arise in general, and the argument breaks down.

\medskip

As a consequence of the above two errors, Lemma~2.11 of \cite{9a} does not hold in its stated form.
Since this lemma is fundamental to the proofs of the main theorems in \cite{9a}, the validity of those results is therefore questionable.
\begin{rem}
	The same oversight described in \emph{Error (b)} also appears in the proof of Lemma~2.8 in \cite{1a}.
	Fortunately, in that case the final conclusion remains valid as it can be recovered independently from Theorem~2I.
\end{rem}

\medskip

In conclusion, the problem of determining the precise form of solutions to the differential equation \eqref{2.1} in the case where $\alpha$ is a common small function of $f$ and $g$ remains unresolved.
This open problem provides the primary motivation for the present work.
In this paper, we focus on identifying explicit solutions of \eqref{2.1} for a particular class of common small functions $\alpha$ associated with $f$ and $g$.
Our results not only fill this gap but also substantially improve and unify the previously known results in the literature.
We now state our main results.

\begin{theo}\label{t2.1}
	Let $f$ and $g$ be two non-constant meromorphic functions having common poles, and let $\alpha$ be a non-zero small function of both $f$ and $g$ such that $0$ and $\infty$ are Borel exceptional values of $\alpha$.
	Assume further that $\rho(\alpha)<\rho(f)$.
	Let $k$ and $n$ be positive integers satisfying $n>k$, and suppose that
	\[
	(f^{n})^{(k)}(g^{n})^{(k)} \equiv \alpha^{2}.
	\]
	
	Then the following conclusions hold.
	
	\begin{enumerate}
		\item[\textnormal{(1)}] \textnormal{If $\rho(\alpha)>0$, then one of the following cases occurs:}
		\begin{enumerate}
			\item[\textnormal{(1A)}]
			\textnormal{$f=R_{1}e^{\delta_{1}}$ and $g=R_{2}e^{\delta_{2}}$, where $R_{1}$ and $R_{2}$ are meromorphic functions satisfying
				\[
				R_{i}=0 \mapsto \alpha=0, \qquad R_{i}=\infty \mapsto \alpha=\infty, \qquad \rho(R_{i})<\rho(\alpha), \quad i=1,2,
				\]
				and $\delta_{1}$ and $\delta_{2}$ are non-constant entire functions of finite order such that $\delta_{1}+\delta_{2}$ is a polynomial and $
				\rho(\alpha)=\deg(\delta_{1}+\delta_{2}).
				$}
			
			\item[\textnormal{(1B)}]
			\textnormal{$\alpha$ reduces to an entire function and $f=c_{1}e^{c\beta}, \,\, g=c_{2}e^{-c\beta},	$
				where
				$
				\beta(z)=\int_{0}^{z}\alpha(t)\,dt,
				$
				and $c$, $c_{1}$, $c_{2}$ are non-zero constants satisfying
				$
				-(nc)^{2}(c_{1}c_{2})^{n}=1.
				$}
		\end{enumerate}
		
		\item[\textnormal{(2)}] \textnormal{If $\rho(\alpha)=0$, then $\alpha$ reduces to a non-zero rational function and}
		$
		f=R_{1}e^{\delta_{1}}, \,\, g=R_{2}e^{-\delta_{1}},
		$
		\textnormal{where $R_{1}$ and $R_{2}$ are non-zero rational functions and $\delta_{1}$ is a non-constant polynomial with $\deg(\delta_{1})=1$, satisfying}
	$2\deg(\alpha^{2})=n\deg(R_{1}R_{2}).$
		\textnormal{In particular, if $\alpha\equiv 1$, then}
		$
		f(z)=c_{1}e^{cz}, \,\,g(z)=c_{2}e^{-cz},
		$
		\textnormal{where $c$, $c_{1}$, and $c_{2}$ are non-zero constants such that}
		$
		(-1)^{k}(c_{1}c_{2})^{n}(nc)^{2k}=1.
		$
	\end{enumerate}
\end{theo}

Clearly Theorem \ref{t2.1} is an improvement of Theorem 2.H.

\begin{theo}\label{t2.2}
	Let $f$ and $g$ be two non-constant meromorphic functions and let $\alpha$ be a non-zero small function of both $f$ and $g$ such that $0$ and $\infty$ are Borel exceptional values of $\alpha$,  $\rho(\alpha)<\rho(f)$ and let $k$ and $n$ be positive integers satisfying  $n>2k$. Suppose
	\[
	(f^n)^{(k)} (g^n)^{(k)} \equiv \alpha^2.
	\]
	
	Then the following cases occur:
	
		\begin{enumerate}
		\item[\em (1)] If $\rho(\alpha)>0$ and $\alpha$ has finitely many zeros and poles, then
	$f = R_1 e^{\delta_1}, \quad g = R_2 e^{\delta_2},$
		where $R_1$, $R_2$ are non-zero rational functions and $\delta_1$, $\delta_2$ are non-constant polynomials satisfying
		\[
		\deg(\delta_1) = \deg(\delta_2), \quad n(\delta_1 + \delta_2) - P = \text{constant}, \quad 2k\,\deg(\delta_1^{(1)}) = 2\deg(\alpha) - n\,\deg(R_1 R_2).
		\]
		
		\item[\em (2)] If $\rho(\alpha)>0$ and $\alpha$ has infinitely many zeros or poles, then one of the following occurs:
		
		\begin{enumerate}
			\item[\em (2A)] $$f = R_1 e^{\delta_1}, \quad g = R_2 e^{\delta_2},$$ where $\overline{\rho}_1(0;R_i)<\rho(\alpha)$, $\overline{\rho}_1(\infty;R_i)<\rho(\alpha)$ for $i=1,2$, $R_1 R_2$ is a small function of $\alpha$, and $\delta_1$, $\delta_2$ are non-constant polynomials with
			\[
			\deg(\delta_1 + \delta_2) = \rho(\alpha).
			\]
			
			\item[\em (2B)] If $\rho_2(f)<+\infty$, then
			\[
			f = R_1 e^{\delta_1}, \quad g = R_2 e^{\delta_2},
			\]
			where $\overline{\rho}_1(0;R_i)<\rho(\alpha)$, $\overline{\rho}_1(\infty;R_i)<\rho_2(f)$ for $i=1,2$, and $\delta_1$, $\delta_2$ are transcendental entire functions such that $\delta_1 + \delta_2$ is a polynomial.
			
			\item[\em (2C)] If $\rho_2(f) = +\infty$, then
			\[
			f = R_1 e^{\delta_1}, \quad g = R_2 e^{\delta_2},
			\]
			where $\overline{\rho}_1(0;R_i)<\rho(\alpha)$, $\overline{N}(r,R_i)=S(r,f)$ for $i=1,2$, and $\delta_1$, $\delta_2$ are transcendental entire functions.
		\end{enumerate}
		
		\item[\em (3)] If $\rho(\alpha) = 0$, then $\alpha$ reduces to a non-zero rational function, and
		\[
		f = R_1 e^{\delta_1}, \quad g = R_2 e^{\delta_2},
		\]
		where $R_1$, $R_2$ are non-zero rational functions and $\delta_1$, $\delta_2$ are non-constant polynomials of degree at most two satisfying
		\[
		\delta_1^{(1)} = - \delta_2^{(1)}, \quad 2k\,\deg(\delta_1^{(1)}) = 2\deg(\alpha) - n\,\deg(R_1 R_2).
		\]
		In particular, if $\alpha \equiv 1$, then
		\[
		f(z) = c_1 e^{cz}, \quad g(z) = c_2 e^{-cz},
		\]
		where $c, c_1, c_2$ are non-zero constants satisfying
		\[
		(-1)^k (c_1 c_2)^n (n c)^{2k} = 1.
		\]
	\end{enumerate}
\end{theo}

\begin{table}[htbp]
	\centering
	\footnotesize
	\renewcommand{\arraystretch}{1.15}
	\begin{tabular}{||
			>{\centering\arraybackslash}p{1.4cm}||
			>{\centering\arraybackslash}p{1.8cm}||
			>{\centering\arraybackslash}p{1.4cm}||
			>{\centering\arraybackslash}p{2.5cm}||
			>{\centering\arraybackslash}p{2.0cm}||
			>{\centering\arraybackslash}p{3.0cm}||
		}
		\hline\hline
		\rowcolor{gray!25}
		\textbf{Theorems}
		& \textbf{Class of $f,g$}
		& \textbf{Common Poles}
		& \textbf{$k,n$}
		& \textbf{RHS $\alpha$}
		& \textbf{Conclusion} \\
		\hline\hline
		
		\rowcolor{blue!2}
		2A \cite{11}
		& Entire
		& No
		& $k=1,\; n\ge1$
		& $1$
		& Exponential \\
		\hline\hline
		
		\rowcolor{blue!6}
		2B \cite{11}
		& Meromorphic
		& No
		& $k=1,\; n\ge6$
		& $1$
		& Exponential \\
		\hline\hline
		
		\rowcolor{blue!10}
		2C \cite{4}
		& Entire
		& No
		& $k=1,\; n=1$
		& $z^2$
		& Exponential / Linear \\
		\hline\hline
		
		\rowcolor{blue!14}
		2D \cite{4}
		& Entire
		& No
		& $k=1,\; n\ge2$
		& $z^2$
		& Exponential \\
		\hline\hline
		
		\rowcolor{blue!18}
		2E \cite{4}
		& Meromorphic
		& No
		& $k=1,\; n\ge4$
		& $z^2$
		& Exponential \\
		\hline\hline
		
		\rowcolor{teal!22}
		2F \cite{3}
		& Entire
		& No
		& $n>k$
		& $1$
		& Exponential \\
		\hline\hline
		
		\rowcolor{teal!26}
		2G \cite{16}
		& Entire
		& No
		& $n>k+2$
		& $z^2$
		& Exponential \\
		\hline\hline
		
		\rowcolor{teal!30}
		2H \cite{14}
		& Meromorphic
		& Yes
		& $n>k+10$
		& $p^2,\;\deg p\le5$
		& $e^{\pm c\int p}$ \\
		\hline\hline
		
		\rowcolor{teal!34}
		2I \cite{7}
		& Meromorphic
		& No
		& $n>2k$
		& $1$
		& Transcendental exponential \\
		\hline\hline
		
		\rowcolor{teal!36}
		2J \cite{8a}
		& Transdental\ meromorphic
		& No
		& $n>\max\{2k,k+2\}$
		& $p^2$
		& $e^{\pm c\int p}$ \\
		\hline\hline
		
		\rowcolor{red!35}
		2K \cite{9a}
		& Meromorphic
		& No
		& $n>2k$
		& Small fn.
		& \textbf{Invalid} \\
		\hline\hline
		
		\rowcolor{orange!35}
		\textbf{2.1}
		& Meromorphic
		& \textbf{Yes}
		& $n>k$
		& Small function, $\rho(\alpha)<\rho(f)$
		& \textbf{Complete classification} \\
		\hline\hline
		
		\rowcolor{orange!40}
		\textbf{2.2}
		& Meromorphic
		& No
		& $n>2k$
		& Small function
		& \textbf{Most general forms} \\
		\hline\hline
	\end{tabular}
	\hspace{8cc}
	\caption{Comparison of solvable cases of $(f^n)^{(k)}(g^n)^{(k)}=\alpha^2$.}
	\label{tab:historical-comparison}
\end{table}

\newpage In the proof of Theorem \ref{t2.1}, we make use of the following key lemmas.

\begin{lem} \label{l2.1}\cite[Lemma 3.5]{5} Let $F$ be a non-constant meromorphic and let $f=\frac{F^{(1)}}{F}$. Then for $n\geq 1$, \beas \frac{F^{(n)}}{F}=f^{n}+\frac{n(n-1)}{2}f^{n-2}f^{(1)}+a_{n}f^{n-3}f^{(2)}+b_{n}f^{n-4}(f^{(1)})^{2}+P_{n-3}(f),\eeas
where $a_{n}=\frac{1}{6}n(n-1)(n-2)$, $b_{n}=\frac{1}{8}n(n-1)(n-2)(n-3)$ and $P_{n-3}(f)$ is a differential polynomial with constant coefficients, which vanishes identically for $n\leq 3$ and has degree $n-3$ when $n>3$.
\end{lem}

\begin{lem}\label{l2.2}\cite{10} Let $f$ be a non-constant meromorphic function and let $a_{n}(\not\equiv 0)$,
$a_{n-1},\ldots, a_{0}$ be small functions of $f$. Then
$T(r,a_{n}f^{n} + a_{n-1}f^{n-1} +\ldots+a_{1}f + a_{0})= nT(r,f) + S(r,f)$.
\end{lem}

\begin{lem}\label{l2.3}\cite{8} Let $f$ be a non-constant meromorphic function and $k\geq 2$ be an integer. If
\[N(r,f)+N(r,0;f)+N(r,0;f^{(k)})=S\left(r,\frac{f^{(1)}}{f}\right),\]
then $f(z)=\exp(az+b)$, where $a\not=0$ and $b$ are constants.
\end{lem}

Using Theorem 1.A, the following lemma can be obtained from the proof of Theorem 1.24 \cite{12}.
\begin{lem}\label{l2.5} \cite[Theorem 1.24]{12} Suppose that $f$ is a non-constant meromorphic function and let $k$ is a positive integer. Then $N(r,0;f^{(k)})\leq N(r,0;f)+k\ol N(r,f)+O(\log T(r,f)+\log r)$
as $r\rightarrow \infty$, outside of a possible exceptional set of finite linear measure.
\end{lem}

Again using Theorem 1.A, the following lemma can be obtained from the proof of Lemma 2.4 \cite{15}.
\begin{lem}\label{l2.6}\cite[Lemma 2.4]{15} Let $f$ be a non-constant meromorphic function and let $k$ and $p$ be positive integers. Then
\beas N_{p}\left(r,0;f^{(k)}\right) \leq T\left(r,f^{(k)}\right)-T(r,f)+ N_{p+k}(r,0;f) + O(\log T(r,f)+\log r)\eeas
and
\beas N_{p}\left(r,0;f^{(k)}\right) \leq k\ol N(r,\infty ;f) + N_{p+k}(r,0;f) +O(\log T(r,f)+\log r)\eeas
as $r\rightarrow \infty$, outside of a possible exceptional set of finite linear measure.
\end{lem}

\begin{lem}\label{l2.6a} \cite{12a} Let $f$ be a non-constant meromorphic function and let $a_{1}$ and $a_{2}$ be two distinct small functions of $f$. Then $T(r,f)\leq \ol N(r,f)+\ol N(r,a_1;f)+\ol N(r,a_2;f)+S(r,f)$.
\end{lem}

\begin{lem}\label{l2.7} \cite{8} Let $f$ be a meromorphic function of infinite order. Then there exists a sequence $\{z_n\}$, $z_{n}\rightarrow \infty$ such that $f^{\#}(z_{n})>|z_{n}|^{N}$ holds for every $N>0$, if $n$ is sufficiently large.
\end{lem}

\begin{lem}\label{l2.8}\cite[Lemmas 1 and 2]{2} Let $f$ be a meromorphic function. If $f$ has bounded spherical derivative on $\mathbb{C}$, then $\rho(f)\leq 2$. If in addition, $f$ has finitely many poles then $\rho(f)\leq 1$.
\end{lem}

\begin{proof}[{\bf Proof of Theorem \ref{t2.1}}]
Suppose
\be\label{al.1} (f^{n})^{(k)}(g^{n})^{(k)}\equiv \alpha^{2}.\ee

Since $\rho(\alpha)<\rho(g)$, from (\ref{al.1}) we get
\beas \rho(f)=\rho(f^{n})=\rho((f^{n})^{(k)})=\rho\left(\frac{(g^{n})^{(k)}}{\alpha^2}\right)=\rho((g^{n})^{(k)})=\rho(g^{n})=\rho(g),\eeas
i.e., $\rho(f)=\rho(g)$. Since $n>k$, using Lemma \ref{l2.6} we get $T(r,f)=O(T(r,(f^n)^{(k)}))$. Also we have $T(r,(f^n)^{(k)})=O(T(r,f))$. Therefore $S(r,f)=S(r,(f^n)^{(k)}))$. Similarly we get $S(r,g)=S(r,(g^n)^{(k)}))$. Again from (\ref{al.1}), we have $S(r,(f^n)^{(k)}))=S(r,(g^n)^{(k)}))$. Therefore $S(r,f)=S(r,g)$.

Note that if $z_{0}$ is a zero of $f$ of multiplicity $p_0$, then $z_{0}$ is also a zero of $(f^{n})^{(k)}$ of multiplicity $np_0-k$. Since $n>k$, then by a routine calculation we can conclude that $z_0$ is a zero of $\alpha$ of multiplicity at least $p_0$. Therefore $f=0\mapsto \alpha=0$. Similarly we can prove that $g=0\mapsto \alpha=0$.

Again if $z_{1}$ is a pole of $f$ of multiplicity $p_1$, then $z_{1}$ is also a pole of $(f^{n})^{(k)}$ of multiplicity $np_1+k$. Since $f$ and $g$ share $\infty$ IM, we can conclude that $z_1$ is a pole of $\alpha$ of multiplicity at least $p_1$. Therefore $f=\infty\mapsto \alpha=\infty$. Similarly we can prove that $g=\infty\mapsto \alpha=\infty$.

Now we consider following two cases.\par
{\bf Case 1.} Let $\rho(\alpha)>0$. It is given that $\rho_1(\alpha)<\rho(\alpha)$ and $\rho_1(\frac{1}{\alpha})<\rho(\alpha)$.
Then by Theorem 1.C we get
\bea\label{al.1.1}\mu(\alpha)=\rho(\alpha)\;\; \text{and}\;\;\rho(\alpha)\in\mathbb{N}.\eea

Set $\alpha=\frac{\pi_1}{\pi_2}\exp(\alpha_2)$, where $\pi_1$ and $\pi_2$ are the canonical products formed with the zeros and poles of $\alpha$ respectively and $\alpha_{2}$ is a polynomial such that $\rho(\alpha)=\deg(\alpha_2)$. Let $\alpha_1=\frac{\pi_1}{\pi_2}$. Then $\alpha=\alpha_1\exp(\alpha_2)$. Noted that $\rho(\alpha_1)<\rho(\alpha)$ and so from (\ref{al.1.1}), we get $\rho(\alpha_1)<\mu(\alpha)$. Then by Theorem 1.B, we get
\bea\label{al.1.0a} T(r,\alpha_1)=S(r,\alpha).\eea

Now $f=0\mapsto \alpha=0$ implies that $N(r,0;f)\leq N(r,0;\alpha)$ and so
\[\rho_1(f)=\varlimsup\limits_{r\to\infty}\frac{\log^+ N(r,0;f)}{\log r}\leq \varlimsup\limits_{r\to\infty}\frac{\log^+ N(r,0;\alpha)}{\log r}=\rho_1(\alpha)<\rho(\alpha),\]
i.e., $\rho_1(f)<\rho(\alpha)$. Similarly we have $\rho_1(g)<\rho(\alpha)$, $\rho_1(\frac{1}{f})<\rho(\alpha)$ and  $\rho_1(\frac{1}{g})<\rho(\alpha)$. Since $\rho(\alpha)<\rho(f)=\rho(g)$, it follows that $0$ and $\infty$ are the Borel exceptional values of both $f$ and $g$. Consequently by Theorem 1.C, we have
$\mu(f)=\rho(f)$, $\mu(g)=\rho(g)$ and $\rho(f),\;\rho(g)$ are positive integers or $\infty$.

Set $f=\frac{R_{11}}{R_{12}}\exp(\delta_1)$, where $R_{11}$ and $R_{12}$ are the canonical products formed with the zeros and poles of $f$ respectively and $\delta_{1}$ is a non-constant entire function. Again let $g=\frac{R_{21}}{R_{22}}\exp(\delta_2)$, where $R_{21}$ and $R_{22}$ are the canonical products formed with the zeros and poles of $g$ respectively and $\delta_{2}$ is a non-constant entire function. Now by Ash \cite[Theorem 4.3.6]{1}, we get $\rho(R_{11})=\rho_1(f)<\rho(\alpha)$, $\rho(R_{12})=\rho_1(\frac{1}{f})<\rho(\alpha)$, $\rho(R_{21})=\rho_1(g)<\rho(\alpha)$ and $\rho(R_{22})=\rho_1(\frac{1}{g})<\rho(\alpha)$.
If we take $R_1=\frac{R_{11}}{R_{12}}$ and $R_2=\frac{R_{21}}{R_{22}}$, then we have
\bea\label{al.1.1a}f=R_1\exp(\delta_1)\;\;\text{and}\;\;g=R_2\exp(\delta_2).\eea

It is clear that $R_i=0\mapsto \alpha=0$ and $R_i=\infty\mapsto \alpha=\infty$ for $i=1,2$. Also we have $\rho(R_1)<\rho(\alpha)$ and $\rho(R_2)<\rho(\alpha)$. Now using Theorem 1.B, we deduce from (\ref{al.1.1}) that $T(r,R_1)=S(r,\alpha)$ and $T(r,R_2)=S(r,\alpha)$. Since $T(r,\alpha)=S(r,f)=S(r,g)$, we get $T(r,R_1)=S(r,f)$ and $T(r,R_2)=S(r,g)$.
Therefore using Lemma \ref{l2.2} to (\ref{al.1.1a}), we get
$T(r,f)=T(r,\exp(\delta_1))+S(r,f)$ and $T(r,g)=T(r,\exp(\delta_2))+S(r,g)$.
Clearly $\rho\big(\frac{(R_1^n)^{(1)}}{R_1^n}\big)< \rho(\alpha)$, $\rho\big(\frac{(R_2^n)^{(1)}}{R_2^n}\big)< \rho(\alpha)$ and $\rho\big((R_1R_2)^n\big)<\rho(\alpha)$. Consequently
\bea\label{al.1.2a}\rho\Big(\frac{\alpha^2}{(R_1R_2)^n}\Big)=\rho(\alpha).\eea

Note that $\frac{(f^n)^{(1)}}{f^n}=\frac{(R_1^n)^{(1)}}{R_1^n}+n\delta_1^{(1)}$ and $\frac{(g^n)^{(1)}}{g^n}=\frac{(R_2^n)^{(1)}}{R_2^n}+n\delta_2^{(1)}$.
Now using Lemma \ref{l2.1}, we conclude from (\ref{al.1}) that
\bea\label{al.1.3} (f^{n})^{(k)}=\left(\left(\frac{(R_1^n)^{(1)}}{R_1^n}+n\delta_1^{(1)}\right)^{k}+P^{*}_{k-1}\left(\frac{(R_1^n)^{(1)}}{R_1^n}+n\delta_1^{(1)}\right)\right)R_{1}^{n}\exp(n\delta_1)\eea
and
\bea\label{al.1.4} (g^{n})^{(k)}=\left(\left(\frac{(R_2^n)^{(1)}}{R_2^n}+n\delta_2^{(1)}\right)^{k}+P^{*}_{k-1}\left(\frac{(R_2^n)^{(1)}}{R_2^n}+n\delta_2^{(1)}\right)\right)R_{2}^{n}\exp(n\delta_2),\eea
where $P^{*}_{k-1}\Big(\frac{(R_1^n)^{(1)}}{R_1^n}+n\delta_1^{(1)}\Big)$ is a differential polynomial of degree at most $k-1$ in $\frac{(R_1^n)^{(1)}}{R_1^n}+n\delta_1^{(1)}$.

Therefore from (\ref{al.1}), (\ref{al.1.3}) and (\ref{al.1.4}), we have
\bea\label{al.1.5} \tilde P\tilde Q \left(R_{1}R_2\right)^{n}\exp(n(\delta_1+\delta_2))\equiv \alpha^{2},\eea
where
\bea\label{al.1.5b}\tilde P=\left(\frac{(R_1^n)^{(1)}}{R_1^n}+n\delta_1^{(1)}\right)^{k}+P^{*}_{k-1}\left(\frac{(R_1^n)^{(1)}}{R_1^n}+n\delta_1^{(1)}\right)\eea
and
\bea\label{al.1.5bb}\tilde Q=\left(\frac{(R_2^n)^{(1)}}{R_2^n}+n\delta_2^{(1)}\right)^{k}+P^{*}_{k-1}\left(\frac{(R_2^n)^{(1)}}{R_2^n}+n\delta_2^{(1)}\right).\eea

Now we want to prove that $(i)$ $\rho(\delta_1)<+\infty$, $\rho(\delta_2)<+\infty$ and $(ii)$ $\delta_1+\delta_2$ is a polynomial.

If not, suppose $\delta=\delta_1+\delta_2$ is a transcendental entire function. Certainly $\rho(\exp(-\delta))=+\infty$. Since $\rho(R_1)<+\infty$, $\rho(R_2)<+\infty$ and $\rho(\alpha)<+\infty$, we conclude that $R_1$, $R_2$ and $\alpha$ are small functions of $\exp(-\delta)$.
We now consider following two cases.\par
{\bf Case A.} Let $k=1$. Then from (\ref{al.1}), we have
\bea\label{alr.1} n^2 f^{n-1}f^{(1)}g^{n-1}g^{(1)}=\alpha^2.\eea

Let $ h=\frac{1}{fg}$.
Clearly from (\ref{al.1.1a}) we get $h=\frac{1}{R_1R_2}\exp(-\delta)$, where $\delta=\delta_1+\delta_2$. Note that
$T(r,h)=T\left(r,\frac{1}{R_1R_2}\exp(-\delta)\right)=T(r,\exp(-\delta))+S(r,\exp(-\delta))$. Consequently $S(r,h)=S(r,\exp(-\delta))$ and so $N(r,h)=S(r,h)$. Now from (\ref{alr.1}), we deduce that
\be\label{alr.3} \left(\frac{g^{(1)}}{g}+\frac{1}{2}\frac{h^{(1)}}{h}\right)^{2}=\frac{1}{4}\left(\frac{h^{(1)}}{h}\right)^{2}-\frac{\alpha^2 h^n}{n^2}.\ee

Let
\bea\label{alr.4} \xi=\frac{g^{(1)}}{g}+\frac{1}{2}\frac{h^{(1)}}{h}.\eea

Then from (\ref{alr.3}) and (\ref{alr.4}), we have
\be\label{alr.5}\xi^2=\frac{1}{4}\left(\frac{h^{(1)}}{h}\right)^{2}-\frac{\alpha^2 h^n}{n^2}.\ee

First we suppose $\xi\equiv 0$.
Then from (\ref{alr.5}), we deduce that
$h^n= \frac{n^2}{4\alpha^2}\left(\frac{h^{(1)}}{h}\right)^{2}$.\\
Since $T(r,\alpha)=S(r,h)$, we deduce that $T(r,h)=S(r,h)$, which is impossible.

Next we suppose $\xi\not\equiv 0$.
Since $T(r,\alpha)=S(r,h)$ and $N(r,h)=S(r,h)$, from (\ref{alr.5}) we have
\bea\label{alr.5a} 2T(r,\xi)=T(r,\xi^2)+S(r,\xi)&=&T\left(r,\frac{1}{4}\left(\frac{h^{(1)}}{h}\right)^{2}-\frac{\alpha^2 h^n}{n^2}\right)+S(r,\xi)\\&=&m\left(r,\frac{1}{4}\left(\frac{h^{(1)}}{h}\right)^{2}-\frac{\alpha^2 h^n}{n^2}\right)+N\left(r,\frac{1}{4}\left(\frac{h^{(1)}}{h}\right)^{2}-\frac{\alpha^2 h^n}{n^2}\right)+S(r,\xi)\nonumber\\&\leq&
n\;m(r,h)+S(r,h)+S(r,\xi)\nonumber\\&\leq&nT(r,h)+S(r,h)+S(r,\xi).\nonumber\eea

This shows that $S(r,\xi)$ can be replaced by $S(r,h)$.
Differentiating (\ref{alr.5}) once, we get
\bea\label{alr.6} 2\xi\xi^{(1)}&=& \frac{1}{2}\frac{h^{(1)}}{h}\left(\frac{h^{(1)}}{h}\right)^{(1)}-\frac{ \alpha^2 h^{n-1}h^{(1)}}{n}-\frac{2 \alpha\alpha^{(1)} h^n}{n^2}\\&=&
\frac{1}{2}\frac{h^{(1)}}{h}\left(\frac{h^{(1)}}{h}\right)^{(1)}-\frac{ \alpha^2 h^{n}}{n^2}\left(n\frac{h^{(1)}}{h}+2\frac{\alpha^{(1)}}{\alpha}\right).\nonumber\eea

Now from (\ref{alr.5}) and (\ref{alr.6}), we get
\be\label{alr.7}\frac{ \alpha^2 h^{n}}{n^2}\left(n\frac{h^{(1)}}{h}+2\frac{\alpha^{(1)}}{\alpha}-2\frac{\xi^{(1)}}{\xi}\right)= \frac{1}{2}\frac{h^{(1)}}{h}\left(\left(\frac{h^{(1)}}{h}\right)^{(1)}-\frac{h^{(1)}}{h}\frac{\xi^{(1)}}{\xi}\right).\ee

Let $n\frac{h^{(1)}}{h}+2\frac{\alpha^{(1)}}{\alpha}-2\frac{\xi^{(1)}}{\xi}\equiv 0$. Since $h$ is transcendental, from (\ref{alr.7}) we get $\left(\frac{h^{(1)}}{h}\right)^{(1)}-\frac{h^{(1)}}{h}\frac{\xi^{(1)}}{\xi}\equiv 0$. On integration we get $\frac{h^{(1)}}{h}=c_0\xi$,
where $c_0$ is a non-zero constant. If $c_0=2$, then from (\ref{alr.4}) we get a contradiction. Hence $c_0\neq 2$ and so from (\ref{alr.4}) we get
$\frac{g^{(1)}}{g}=\frac{1-c_0}{c_0}\frac{h^{(1)}}{h}$.
Now from (\ref{alr.5}), we get
\bea\label{alr.9} h^n=\frac{n^2}{ \alpha^2}\left(\frac{1}{4}-\frac{1}{c_0^2}\right)\left(\frac{h^{(1)}}{h}\right)^{(1)}.\eea
Since $T(r,\alpha)=S(r,h)$, from (\ref{alr.9}) we conclude that $T(r,h)=S(r,h)$, which is impossible.

Let $n\frac{h^{(1)}}{h}+2\frac{\alpha^{(1)}}{\alpha}-2\frac{\xi^{(1)}}{\xi}\not\equiv 0$.
Note that $N\big(r,\frac{h^{(1)}}{h}\big)=S(r,h)$. Since $g=R_2\exp(\delta_2)$ and $T(r,R_2)=S(r,h)$, it follows that $N(r,\xi)=S(r,h)$. Now from (\ref{alr.7}), we get
\bea\label{alr.10} nT(r,h)&=&m(r,h^n)+S(r,h)\\&\leq& m\left(r,\frac{1}{2}\frac{h^{(1)}}{h}\left(\left(\frac{h^{(1)}}{h}\right)^{(1)}-\frac{h^{(1)}}{h}\frac{\xi^{(1)}}{\xi}\right)\right)\nonumber\\&&+m\left(r,\frac{1}{\alpha^2\left(n\frac{h^{(1)}}{h}+2\frac{\alpha^{(1)}}{\alpha}-2\frac{\xi^{(1)}}{\xi}\right)}\right)+S(r,h)\nonumber\\&\leq&m\left(r,\frac{1}{2}\frac{h^{(1)}}{h}\left(\left(\frac{h^{(1)}}{h}\right)^{(1)}-\frac{h^{(1)}}{h}\frac{\xi^{(1)}}{\xi}\right)\right)+m\left(r,\alpha^2\left(n\frac{h^{(1)}}{h}+2\frac{\alpha^{(1)}}{\alpha}-2\frac{\xi^{(1)}}{\xi}\right)\right)\nonumber\\&&+N\left(r,\alpha^2\left(n\frac{h^{(1)}}{h}+2\frac{\alpha^{(1)}}{\alpha}-2\frac{\xi^{(1)}}{\xi}\right)\right)+S(r,h)\nonumber\\&\leq& N\Big(r,\frac{\xi^{(1)}}{\xi}\Big)+S(r,h)+S(r,\xi)\leq T(r,\xi)+S(r,h)+S(r,\xi).\nonumber
\eea

Since $S(r,\xi)\leq S(r,h)$, from (\ref{alr.5a}) and (\ref{alr.10}) we get $\frac{n}{2}T(r,h)= S(r,h)$, which is impossible.

Hence $\delta$ is a polynomial. Clearly $\delta_2^{(1)}=\delta^{(1)}-\delta_1^{(1)}$. Next we want to prove that $\rho(\delta_1)<+\infty$ and $\rho(\delta_2)<+\infty$. If not, suppose $\rho(\delta_1)=+\infty$.
Now from (\ref{al.1.5}), (\ref{al.1.5b}) and (\ref{al.1.5bb}), we have
\bea\label{alr.11} \Big(\frac{(R_1^n)^{(1)}}{R_1^n}+n\delta_1^{(1)}\Big)\Big(\frac{(R_2^n)^{(1)}}{R_2^n}+n\delta^{(1)}-n\delta_1^{(1)}\Big)\exp(n\delta)=\frac{\alpha^{2}}{(R_{1}R_{2})^{n}}.\eea

Note that $\frac{(R_1^n)^{(1)}}{R_1^n}$ and $\frac{(R_2^n)^{(1)}}{R_2^n}+n\delta^{(1)}$ are the small functions of $\delta_1^{(1)}$. Now using Lemma \ref{l2.2}, we get
\[T\left(r,\Big(\frac{(R_1^n)^{(1)}}{R_1^n}+n\delta_1^{(1)}\Big)\Big(\frac{(R_2^n)^{(1)}}{R_2^n}+n\delta^{(1)}-n\delta_1^{(1)}\Big)\right)=2T(r,\delta_1^{(1)})+S(r,\delta_1^{(1)}),\]
which shows that $\big(\frac{(R_1^n)^{(1)}}{R_1^n}+n\delta_1^{(1)}\big)\big(\frac{(R_2^n)^{(1)}}{R_2^n}+n\delta^{(1)}-n\delta_1^{(1)}\big)$ is of infinite order. Also $\exp(n\delta)$ is of finite order. Then from (\ref{al.1.2a}) and (\ref{alr.11}), we get a contradiction. Hence $\rho(\delta_1)<+\infty$ and $\rho(\delta_2)<+\infty$.\par

{\bf Case B.} Let $k\geq 2$. We want to prove that $\rho(\delta_1)<+\infty$. If not, suppose $\rho(\delta_1)=+\infty$. Since $\rho\big(\frac{(R_1^n)^{(1)}}{R_1^n}\big)<+\infty$, it follows that $\frac{(R_1^n)^{(1)}}{R_1^n}$ is a small function of $\delta_1^{(1)}$. Therefore by Lemma \ref{l2.2}, we get
\[T\Big(r,\frac{(R_1^n)^{(1)}}{R_1^n}+n\delta_1^{(1)}\Big)=T(r,\delta_1^{(1)})+S(r,\delta_1^{(1)})\]
and so $\frac{(R_1^n)^{(1)}}{R_1^n}+n\delta_1^{(1)}$ is of infinite order. Since $\alpha$ is of finite order, we conclude that
\bea\label{alr.12a} T(r,\alpha)=S\Big(r,\frac{(R_1^n)^{(1)}}{R_1^n}+n\delta_1^{(1)}\Big).\eea

Note that $S\big(r,\frac{(R_1^n)^{(1)}}{R_1^n}+n\delta_1^{(1)}\big)=S\big(r,\frac{(f^n)^{(1)}}{f^n}\big)$.
Since $R_1=0\mapsto \alpha=0$ and $R_1=\infty\mapsto \alpha=\infty$, it follows from (\ref{al.1.1a}) that
\bea\label{alr.12} N(r,0;f^n)=N(r,0;R_1^n)\leq N(r,0,\alpha^n)\leq nT(r,\alpha).\eea
and
\bea\label{alr.12b} N(r,f^n)=N(r,R_1^n)\leq N(r,\alpha^n)\leq nT(r,\alpha).\eea

Again from (\ref{al.1}), we have
\bea\label{alr.13} N(r,0;(f^{n})^{(k)})\leq N(r,0;\alpha^{2})\leq 2T(r,\alpha).\eea

Consequently from (\ref{alr.12a})-(\ref{alr.13}), we have
\bea\label{alr.14} N(r,f^{n})+N(r,0;f^{n})+N(r,0;(f^{n})^{(k)})&\leq& 2(n+1)T(r,\alpha)\\&=&S\Big(r,\frac{(R_1^n)^{(1)}}{R_1^n}+n\delta_1^{(1)}\Big)=S\Big(r,\frac{(f^{n})^{(1)}}{f^{n}}\Big).\nonumber\eea

Then using Lemma \ref{l2.3} to (\ref{alr.14}) we conclude that $f(z)=\exp(az+b)$, where $a\not=0$ and $b$ are constants. Consequently $\delta_1$ is a polynomial, which contradicts the fact that $\rho(\delta_1)=+\infty$.

Hence $\rho(\delta_1)<+\infty$. Similarly we can prove that $\rho(\delta_2)<+\infty$. Since $\delta=\delta_1+\delta_2$, from (\ref{al.1.5}) we get
\bea\label{alr.15} \tilde P\tilde Q \exp(n\delta)\equiv \frac{\alpha^{2}}{\left(R_{1}R_2\right)^{n}}.\eea

We now prove that $\delta$ is a polynomial. If not, suppose $\delta$ is a transcendental entire function. Clearly $\rho(\exp(n\delta))=+\infty$. On the other hand $\rho\big(\frac{(R_1^n)^{(1)}}{R_1^n}+n\delta_1^{(1)}\big)<+\infty$ and $\rho\big(\frac{(R_2^n)^{(1)}}{R_2^n}+n\delta_2^{(1)}\big)<+\infty$.

Since $T(r, f^{(i)})\leq (i+1)T(r,f)+S(r,f)$, by simple calculations we derive from
(\ref{al.1.5b}) and (\ref{al.1.5bb}) that
\bea\label{alr.15a}T(r,\tilde P)\leq K_1 T\Big(r, \frac{(R_1^n)^{(1)}}{R_1^n}+n\delta_1^{(1)}\Big)+S\Big(r, \frac{(R_1^n)^{(1)}}{R_1^n}+n\delta_1^{(1)}\Big)\eea
and
\bea\label{alr.15aa}T(r,\tilde Q)\leq K_2 T\Big(r, \frac{(R_2^n)^{(1)}}{R_2^n}+n\delta_2^{(1)}\Big)+S\Big(r, \frac{(R_2^n)^{(1)}}{R_2^n}+n\delta_2^{(1)}\Big)\eea
respectively, where $K_1$ and $K_2$ are fixed positive integers. Consequently
\bea\label{alr.16}\rho(\tilde P)\leq\rho\Big(\frac{(R_1^n)^{(1)}}{R_1^n}+n\delta_1^{(1)}\Big)<+\infty\;\;
\text{and}\;\;\rho(\tilde Q)\leq\rho\Big(\frac{(R_2^n)^{(1)}}{R_2^n}+n\delta_2^{(1)}\Big)<+\infty.\eea

Since $\rho(\exp(n\delta))=+\infty$, we get a contradiction from (\ref{al.1.2a}), (\ref{alr.15}) and (\ref{alr.16}).
So $\delta$ is a polynomial.
\vspace{0.1in}

Finally by Cases A and B we prove that $\rho(\delta_1)<+\infty$, $\rho(\delta_2)<+\infty$ and $\delta=\delta_1+\delta_2$ is a polynomial.

We now consider following two sub-cases.\par

{\bf Sub-case 1.1.} Let $\rho(f)<+\infty$. Then from (\ref{al.1.1a}), we deduce that $\delta_1$ and $\delta_2$ are both polynomials. Since $\rho(\alpha)<+\infty$, we have $T(r,\delta_1^{(1)})=S(r,\alpha)=T(r,\delta_2^{(1)})$. Again since $T(r,R_1)=S(r,\alpha)=T(r,R_2)$, we get $T\big(r,\frac{(R_1^n)^{(1)}}{R_1^n}+n\delta_1^{(1)}\big)=S(r,\alpha)$ and $T\big(r,\frac{(R_2^n)^{(1)}}{R_2^n}+n\delta_2^{(1)}\big)=S(r,\alpha)$. Then from (\ref{alr.15a}) and (\ref{alr.15aa}), we see that $T(r,\tilde P)=S(r,\alpha)=T(r,\tilde Q)$. Consequently from (\ref{alr.15}), we get $\deg(\delta)=\rho(\alpha)$.
Therefore $f=R_1\exp(\delta_1)$ and $g=R_2\exp(\delta_2)$, where $R_1$ and $R_2$ are meromorphic functions such that $R_i=0\mapsto \alpha=0$, $R_i=\infty\mapsto \alpha=\infty$ and $\rho(R_i)<\rho(\alpha)$ for $i=1,2$ and $\delta_1$, $\delta_2$ are polynomials such that $\rho(\alpha)=\deg(\delta_1+\delta_2)$.\par

{\bf Sub-case 1.2.} Let $\rho(f)=+\infty$. Clearly $\delta_1$ and $\delta_2$ are transcendental entire functions. Since $\delta_1+\delta_2=\delta$, we have $T(r,\delta_1)=T(r,\delta_2)+S(r,\delta_2)$ and so $\rho(\delta_1)=\rho(\delta_2)$.
Note that $N(r,\tilde P)\leq N(r,0;R_1^n)$, $N(r,\tilde Q)\leq N(r,0;R_2^n)$ and so $N(r,\tilde P)\leq N(r,0;\alpha^n)$ and $N(r,\tilde Q)\leq N(r,0;\alpha^n)$. Consequently
$\rho_1\big(\frac{1}{\tilde P}\big)\leq \rho_1(\alpha^n)=\rho_1(\alpha)$ and $\rho_1\big(\frac{1}{\tilde Q}\big)\leq \rho_1(\alpha^n)=\rho_1(\alpha)$. Since $\rho_1(\alpha)<\rho(\alpha)$, we get
\bea\label{alr.19}\rho_1\big(\frac{1}{\tilde P}\big)<\rho(\alpha)\;\;\text{and}\;\;\rho_1\big(\frac{1}{\tilde Q}\big)<\rho(\alpha).\eea

Again from (\ref{al.1.5}) we have $N(r,0;\tilde P)\leq N(r,0;\alpha^n)$ and $N(r,0;\tilde Q)\leq N(r,0;\alpha^n)$.
Consequently
\bea\label{alr.20}\rho_1(\tilde P)\leq \rho_1(\alpha^n)=\rho_1(\alpha)<\rho(\alpha)\;\;\text{and}\;\;\rho_1(\tilde Q)\leq \rho_1(\alpha^n)=\rho_1(\alpha)<\rho(\alpha).\eea

If $\rho(\alpha)=\deg(\delta)$, then from (\ref{al.1.1a}), we get $f=R_1\exp(\delta_1)$ and $g=R_2\exp(\delta_2)$, where $R_1$ and $R_2$ are meromorphic functions such that $R_i=0\mapsto \alpha=0$, $R_i=\infty\mapsto \alpha=\infty$ and $\rho(R_i)<\rho(\alpha)$ for $i=1,2$ and $\delta_1$, $\delta_2$ are finite order transcendental entire functions such that $\delta_1+\delta_2$ is a polynomial and $\rho(\alpha)=\deg(\delta_1+\delta_2)$.

Next we consider following two sub-cases.\par
{\bf Sub-case 1.2.1.} Let $\rho(\alpha)<\deg(\delta)$. Clearly $\rho(\alpha)<\mu(\exp(\delta))$ and so by Theorem 1.B, we conclude that $T(r,\alpha)=S(r,\exp(\delta))$. Obviously $S(r,\alpha)$ can be replaced by $S(r,\exp(\delta))$.

Now we consider the following two sub-cases.\par

{\bf Sub-case 1.2.1.1.} Let $\rho(\alpha)\geq\rho(\delta_1^{(1)})$.
Since $\rho(\alpha)<\rho(\exp(\delta))$, we have $\rho(\delta_1^{(1)})=\rho(\delta_2^{(1)})<\mu(\exp(\delta))$ and so by Theorem 1.B, we get $T(r,\delta_1^{(1)})=S(r,\exp(\delta))=T(r,\delta_2^{(1)})$. Also we have $T(r,\alpha)=S(r,\exp(\delta))$. In this case we can easily prove that $T(r,\tilde P)=S(r,\exp(\delta))=T(r,\tilde Q)$. Consequently from (\ref{alr.15}), we get $T(r,\exp(\delta))=S(r,\exp(\delta))$, which is impossible.\par
{\bf Sub-case 1.2.1.2.} Let $\rho(\alpha)<\rho(\delta_1^{(1)})$.

We want to prove that $\rho(\alpha)\leq \rho(\tilde P)$. If possible suppose $\rho(\alpha)>\rho(\tilde P)$. Then since $\mu(\alpha)=\rho(\alpha)$, we have $\rho(\tilde P)<\mu(\alpha)$ and so by Theorem 1.B, we get $T(r,\tilde P)=S(r,\alpha)$.

First we suppose $k=1$. In this case, we have $T\big(r,\frac{(R_1^n)^{(1)}}{R_1^n}+n\delta_1^{(1)}\big)=S(r,\alpha)$.
Note that
\[T(r,n\delta_1^{(1)})\leq T\Big(r,\frac{(R_1^n)^{(1)}}{R_1^n}+n\delta_1^{(1)}\Big)+T\Big(r,\frac{(R_1^n)^{(1)}}{R_1^n}\Big)+O(1)\leq T(r,\alpha)+S(r,\alpha),\]
and so $\rho(\delta_1)\leq \rho(\alpha)$, which is impossible.

Next we suppose $k\geq 2$.
If $\rho(\alpha)\geq\rho(\tilde Q)$, then $\rho(\tilde Q)\leq \rho(\alpha)<\rho(\exp(\delta))=\mu(\exp(\delta))$ and so by Theorem 1.B, we obtain $T(r,\tilde Q)=S(r,\exp(\delta))$. In this case also we get at a contradiction.
Hence $\rho(\alpha)<\rho(\tilde Q)$. Now from (\ref{alr.19}) and (\ref{alr.20}), we see that $0$ and $\infty$ are the Borel exceptional values of $\tilde Q$. Now Theorem 1.C gives $\mu(\tilde Q)=\rho(\tilde Q)$ and so $\rho(\alpha)<\mu(\tilde Q)$. Therefore by Theorem 1.B, we get $T(r,\alpha)=S(r,\tilde Q)$.
Now proceeding in the same as done in Case B, we get
$N(r,g^{n})+N(r,0;g^{n})+N(r,0;(g^{n})^{(k)})\leq 2(n+1)T(r,\alpha)$
and so
\bea\label{alr.22} N(r,g^{n})+N(r,0;g^{n})+N(r,0;(g^{n})^{(k)})=S(r,\tilde Q).\eea

Again from (\ref{alr.15aa}), we see that $S(r,\tilde Q)\leq S\Big(r,\frac{(R_2^n)^{(1)}}{R_2^n}+n\delta_2^{(1)}\Big)$.
Then from (\ref{alr.22}), we get
\bea\label{alr.24} N(r,g^{n})+N(r,0;g^{n})+N(r,0;(g^{n})^{(k)})=S\Big(r,\frac{(R_2^n)^{(1)}}{R_2^n}+n\delta_2^{(1)}\Big)=S\Big(r,\frac{(g^{n})^{(1)}}{g^{n}}\Big).\eea

Then using Lemma \ref{l2.3} to (\ref{alr.24}) we get $g(z)=\exp(cz+d)$, where $c\not=0$ and $d$ are constants. Therefore we get a contradiction, since $\rho(g)=+\infty$.

Hence $\rho(\alpha)\leq \rho(\tilde P)$. Similarly we can prove that $\rho(\alpha)\leq \rho(\tilde Q)$.

Suppose $\rho(\alpha)=\rho(\tilde P)$. Then from above we have $T(r,\tilde P)=S(r,\exp(\delta))$. If $\rho(\alpha)=\rho(\tilde Q)$, then we also have $T(r,\tilde Q)=S(r,\exp(\delta))$. Since $T(r,\alpha)=S(r,\exp(\delta))$, from (\ref{alr.15}) we get $T(r,\exp(\delta))=S(r,\exp(\delta))$, which is impossible. If $\rho(\alpha)<\rho(\tilde Q)$, then from above we can easily deduce that $g(z)=\exp(cz+d)$, where $c\not=0$ and $d$ are constants, which is impossible.
Hence $\rho(\alpha)<\rho(\tilde P)$. Similarly $\rho(\alpha)<\rho(\tilde Q)$. In this case also from above we have
\bea\label{alr.25}T(r,\alpha)=S(r,\tilde P)\;\;\text{and}\;\;T(r,\alpha)=S(r,\tilde Q).\eea

If $k\geq 2$, then as above we get $f(z)=\exp(az+b)$ and $g(z)=\exp(cz+d)$, where $a,c\not=0$ and $b,d$ are constants, which is impossible. Hence $k=1$ and so from (\ref{alr.25}) we get
$T(r,\alpha)=S\big(r,\frac{(R_1^n)^{(1)}}{R_1^n}+n\delta_1^{(1)}\big)$. Clearly $S(r,\alpha)\leq S\big(r,\frac{(R_1^n)^{(1)}}{R_1^n}+n\delta_1^{(1)}\big)$.
Since $T(r,R_1)=S(r,\alpha)$, we can deduce that
\bea\label{alr.27}T\Big(r,\frac{(R_1^n)^{(1)}}{R_1^n}+n\delta_1^{(1)}\Big)\leq T(r,n\delta_1^{(1)})+S\Big(r,\frac{(R_1^n)^{(1)}}{R_1^n}+n\delta_1^{(1)}\Big)\eea
and
\bea\label{alr.28}T(r,n\delta_1^{(1)})\leq T\Big(r,\frac{(R_1^n)^{(1)}}{R_1^n}+n\delta_1^{(1)}\Big)+S\Big(r,\frac{(R_1^n)^{(1)}}{R_1^n}+n\delta_1^{(1)}\Big).\eea

Clearly from (\ref{alr.27}) and (\ref{alr.28}) we get
$S(r,n\delta_1^{(1)})=S\big(r,\frac{(R_1^n)^{(1)}}{R_1^n}+n\delta_1^{(1)}\big)$.

Since $T\big(r,\frac{(R_2^n)^{(1)}}{R_2^n}\big)=S(r,\alpha)$, it follows that $T\big(r,\frac{(R_1^n)^{(1)}}{R_1^n}\big)=S(r,n\delta_1^{(1)})$.
Similarly $T\big(r,\frac{(R_2^n)^{(1)}}{R_2^n}\big)=S(r,n\delta_2^{(1)})$. Since $\delta_1+\delta_2=\delta$, we get $S(r,n\delta_1^{(1)})=S(r,n\delta_2^{(1)})$. So $T\big(r,\frac{(R_2^n)^{(1)}}{R_2^n}+n\delta^{(1)}\big)=S(r,n\delta_1^{(1)})$.

Let $\frac{(R_1^n)^{(1)}}{R_1^n}+\frac{(R_2^n)^{(1)}}{R_2^n}+n\delta^{(1)}\equiv 0$. On integration, we get $(R_1R_2)^n=c_0\exp(-n\delta)$, where $c_0$ is a non-zero constant. Then from (\ref{alr.11}) we get $-c_0\big(\frac{(R_1^n)'}{R_1^n}+n\delta_1^{(1)}\big)^2=\alpha^2$, which shows that $\rho(\delta_1^{(1)})=\rho(\alpha)$. Therefore we get a contradiction.

Let $\frac{(R_1^n)^{(1)}}{R_1^n}+\frac{(R_2^n)^{(1)}}{R_2^n}+n\delta^{(1)}\not\equiv 0$. Now using Lemma \ref{l2.6a}, we get
\beas T(r,\delta_1^{(1)})&\leq& \ol N(r,\delta_1^{(1)})+\ol N\Big(r,0;\frac{(R_1^n)^{(1)}}{R_1^n}+n\delta_1^{(1)}\Big)+\ol N\Big(r,0;\frac{(R_2^n)^{(1)}}{R_2^n}+n\delta^{(1)}-n\delta_1^{(1)}\Big)+S(r,\delta_1^{(1)})\\&\leq& 2T(r,\alpha_1)+S(r,\delta_1^{(1)}),\eeas
which shows that $\rho(\delta_1^{(1)})\leq \rho(\alpha_1)<\rho(\alpha)$, which is impossible.\par

{\bf Sub-case 1.2.2.} Let $\rho(\alpha)>\deg(\delta)$. Clearly $\rho(\exp(\delta))<\mu(\alpha)$ and so $T(r,\exp(\delta))=S(r,\alpha)$.
Since $\delta_1+\delta_2=\delta$ and $T(r,R_1)=S(r,\alpha)=T(r,R_2)$, we have
\bea\label{alr.30} T\Big(r,\frac{(R_2^n)^{(1)}}{R_2^n}+n\delta_2^{(1)}\Big)=T\Big(r,-\frac{(R_2^n)^{(1)}}{R_2^n}+n\delta_1^{(1)}\Big)+O(\log r)\leq
T\Big(r,\frac{(R_1^n)^{(1)}}{R_1^n}+n\delta_1^{(1)}\Big)+S(r,\alpha).\eea

Similarly we have
\bea\label{alr.30a} T\Big(r,\frac{(R_1^n)^{(1)}}{R_1^n}+n\delta_1^{(1)}\Big)\leq
T\Big(r,\frac{(R_2^n)^{(1)}}{R_2^n}+n\delta_2^{(1)}\Big)+S(r,\alpha).\eea

Then using Lemma \ref{l2.2}, (\ref{alr.15})-(\ref{alr.15aa}) and (\ref{alr.30})-(\ref{alr.30a}) we get
\beas 2T(r,\alpha)+S(r,\alpha)\leq
(K_1+2K_2)T\Big(r,\frac{(R_1^n)^{(1)}}{R_1^n}+n\delta_1^{(1)}\Big)+S\Big(r, \frac{(R_1^n)^{(1)}}{R_1^n}+n\delta_1^{(1)}\Big)+S(r,\alpha)
\eeas
and so $S(r,\alpha)$ can be replaced by $S\big(r,\frac{(R_1^n)^{(1)}}{R_1^n}+n\delta_1^{(1)}\big)$.

Suppose $k\geq 2$. Then proceeding in the same way as done in Case B, we get
$N(r,f^{n})+N(r,0;f^{n})+N(r,0;(f^{n})^{(k)})\leq 2(n+1)T(r,\alpha_1)$
and so from (\ref{al.1.0a}) we get
\beas N(r,f^{n})+N(r,0;f^{n})+N(r,0;(f^{n})^{(k)})=S\Big(r,\frac{(f^{n})^{(1)}}{f^{n}}\Big).\eeas

Now using Lemma \ref{l2.3}, we get $f(z)=\exp(az+b)$, where $a\not=0$ and $b$ are constants. Therefore we get a contradiction.

Suppose $k=1$. Then proceeding in the same way as done in sub-case 1.2.1.2, we can conclude that $T\big(r,\frac{(R_2^n)^{(1)}}{R_2^n}+n\delta^{(1)}\big)=S(r,n\delta_1^{(1)})=T\big(r,\frac{(R_1^n)^{(1)}}{R_1^n}\big)$.

Let $\frac{(R_1^n)^{(1)}}{R_1^n}+\frac{(R_2^n)^{(1)}}{R_2^n}+n\delta^{(1)}\equiv 0$. On integration, we get $(R_1R_2)^n=c_0\exp(-n\delta)$, which shows that $R_1$ and $R_2$ have no zeros and poles. Consequently $R_1$ and $R_2$ are non-zero constants and so $\delta$ is a constant. Now from (\ref{alr.11}) we get $-c_0\big(n\delta_1^{(1)}\big)^2=\alpha^2$. Let us take $\delta_1^{(1)}=c\alpha$, where $c$ is a non-zero constant. On integration we get $\delta_1(z)=c\int_{0}^z \alpha(t)dt+d_1$, where $d_1$ is a constant. Since $\delta_1+\delta_2=\delta$, we take $\delta_2(z)=-c\int_{0}^z \alpha(t)dt+d_2$, where $d_2$ is a constant. Finally we take $f=c_1\exp(c\beta)$ and $g=c_2\exp(-c\beta)$, where $\beta(z)=\int_{0}^z \alpha(t)dt$, $c$, $c_1$ and $c_2$ are non-zero constants such that $-(nc)^2(c_1c_2)^n=1$.

Let $\frac{(R_1^n)^{(1)}}{R_1^n}+\frac{(R_2^n)^{(1)}}{R_2^n}+n\delta^{(1)}\not\equiv 0$. Then proceeding in the same way as done in sub-case 1.2.1.2, we can conclude that $\rho(\delta_1)\leq \rho(\alpha_1)<\rho(\alpha)$. Since $\rho(\alpha)=\mu(\alpha)$, we have $\rho(\delta_1)<\mu(\alpha)$ and so by Theorem 1.B, we get $T(r,\delta_1)=S(r,\alpha)$. Therefore from (\ref{alr.11}) we get $T(r,\alpha)=S(r,\alpha)$, which is impossible.\par

{\bf Case 2.} Let $\rho(\alpha)=0$. In this case $\alpha$ has only finitely many zeros and poles and so $\alpha$ is a non-zero rational function. Therefore $f=R_1\exp(\delta_1)$ and $g=R_2\exp(\delta_2)$, where $R_1$, $R_2$ are two non-zero rational functions and $\delta_1$, $\delta_2$ are two non-constant entire functions.

Suppose $F=\frac{f^n}{\alpha}$ and $G=\frac{g^n}{\alpha}$.
Let $\mathcal{F}=\{F_{\omega}\}$ and $\mathcal{G}=\{G_{\omega}\}$, where $F_{\omega}(z)=F(z+\omega)$ and $G_{\omega}(z)=G(z+\omega)$, $z\in \mathbb{C}$. Clearly $\mathcal{F}$ and $\mathcal{G}$ are two families of meromorphic functions defined on $\mathbb{C}$.\par

Now we consider following sub-cases.\par
{\bf Sub-case 2.1.} Let $\mathcal{F}$ be normal on $\mathbb{C}$. Then by Theorem 1.D we have $F^{\#}(\omega)=F^{\#}_{\omega}(0)\leq M$ for some $M>0$ and for all $\omega\in\mathbb{C}$. Hence by Lemma \ref{l2.8} we get $\rho(F)\leq 1$. Since $\rho(\alpha)=0$, we have $\rho(f)=1$ and so $\rho(g)=1$. Clearly $\deg(\delta_1)=1$ and $\deg(\delta_2)=1$. Also $\tilde P$ and $\tilde Q$ are non-zero rational functions. Now from (\ref{al.1.5}), we see that $\delta_1+\delta_2$ is a constant, say $c$ and so $\delta_1^{(1)}=-\delta_2^{(1)}$. Then from (\ref{al.1.5})-(\ref{al.1.5bb}), we get
\bea\label{xr.3}&&\left(\left(\frac{(R_1^n)^{(1)}}{R_1^n}+n\delta_1^{(1)}\right)^{k}+P^{*}_{k-1}\left(\frac{(R_1^n)^{(1)}}{R_1^n}+n\delta_1^{(1)}\right)\right)\left(\left(\frac{(R_2^n)^{(1)}}{R_2^n}-n\delta_1^{(1)}\right)^{k}+P^{*}_{k-1}\left(\frac{(R_2^n)^{(1)}}{R_2^n}-n\delta_1^{(1)}\right)\right)\nonumber\\&&=\frac{\exp(-nc)\alpha^2}{(R_1R_2)^{n}}.\eea

Letting $|z|\to \infty$, we get from (\ref{xr.3}) that $ 2\deg(\alpha)=n\deg(R_1R_2)$.
Finally $f=R_1\exp(\delta_1)$ and $g=R_2\exp(-\delta_1)$, where $R_{1}$, $R_{2}$ are non-zero rational functions and $\delta_1$ is non-constant polynomial such that $\deg(\delta_1)=1$ and $ 2\deg(\alpha)=n\deg(R_1R_2)$. In particular if $\alpha\equiv 1$, then $f(z)=c_1\exp(cz)$ and $g(z)=c_2\exp(-cz)$, where $c$, $c_1$ and $c_2$ are non-zero constants such that $(-1)^k(c_1c_2)^n(nc)^{2k}=1$.\par

{\bf Sub-case 2.2.} Let $\mathcal{F}$ be not normal on $\mathbb{C}$. Then there exists at least one $z_{0}\in \Delta$ such that $\mathcal{F}$ is not normal at $z_{0}$. For the sake of simplicity we assume that $z_{0}=0$. Now by Theorem 1.D, there exists a sequence of meromorphic functions $\{F(z+\omega_{j})\}\subset \mathcal{F}$, where $z\in \Delta$ and $\{\omega_{j}\}$ is a sequence of complex numbers such that $F^{\#}(\omega_{j})\rightarrow \infty$ as $\omega_{j}\rightarrow \infty$.
Since $\alpha$ have finitely many zeros and poles, there exists a $r>0$ such that $\alpha\not=0, \infty$ in $D=\{z: |z|\geq r\}$.
We know that $R_i=0\mapsto \alpha=0$ and $R_i=\infty\mapsto \alpha=\infty$ for $i=1,2$. Therefore $F\neq 0, \infty$ and $G\neq 0, \infty$ in $D=\{z: |z|\geq r\}$. As $\omega_{j}\rightarrow \infty$ as $j\rightarrow \infty$, without loss of generality we may assume that $|\omega_{j}|\geq r+1$ for all $j$.
Note that $|\omega_{j}+z|\geq |\omega_{j}|-|z|$ and so $\omega_{j}+z\in D$ for all $z\in \Delta$. Consequently $\{F(z+\omega_{j})\}$ and $\{G(z+\omega_{j})\}$ are two families of analytic functions in $z\in\Delta$ having no zeros. Now by Lemma 1.A, there exist
\begin{enumerate}
\item[(i)] a sequence $\{z_j\}\subset\Delta$ with $z_j\rightarrow 0$ as $j\rightarrow \infty$;
\item[(ii)] a sequence $\{\rho_j\}$ of positive numbers with $\rho_j\rightarrow 0^{+}$ as $j\rightarrow \infty$;
\item[(iii)] a subsequence $\{F(\omega_j+z_{j}+\rho_{j}\zeta)\}$ of $\{F(\omega_j+z)\}$
\end{enumerate}
such that
\bea\label{xa.4} h_{j}(\zeta)=\rho_j^{-k}F(\omega_j+z_{j}+\rho_{j}\zeta)= \rho_j^{-k}\frac{f^n(\omega_{j}+z_{j}+\rho_{j}\zeta)}{\alpha(\omega_{j}+z_{j}+\rho_{j}\zeta)}\rightarrow h(\zeta)\eea
spherically locally uniformly in $\mathbb{C}$, where $h$ is a non-constant entire function having no zeros such that $h^{\#}(\zeta)\leq h^{\#}(0)=1$. Clearly by Lemma \ref{l2.8} we have $\rho(h)\leq 1$.
It is also clear that
\bea\label{xa.6} \frac{\alpha^{(1)}(\omega_{j}+z_{j}+\rho_{j}\zeta)}{\alpha(\omega_{j}+z_{j}+\rho_{j}\zeta)}\rightarrow 0\eea
as $j\rightarrow \infty$. Now by a simple calculation we deduce from (\ref{xa.4}) that
\bea\label{xa.7} \rho_{j}^{-k+1}\frac{(f^n)^{(1)}(\omega_{j}+z_{j}+\rho_{j}\zeta)}{\alpha(\omega_{j}+z_{j}+\rho_{j}\zeta)}&=&h_{j}^{(1)}(\zeta)+\rho_{j}^{-k+1}\frac{\alpha^{(1)}(\omega_{j}+z_{j}+\rho_{j}\zeta)}{\alpha^{2}(\omega_{j}+z_{j}+\rho_{j}\zeta)}f^n(\omega_{j}+z_{j}+\rho_{j}\zeta)\\&=&h_{j}^{(1)}(\zeta)+\rho_{j}\frac{\alpha^{(1)}(\omega_{j}+z_{j}+\rho_{j}\zeta)}{\alpha(\omega_{j}+z_{j}+\rho_{j}\zeta)}h_{j}(\zeta).\nonumber \eea

Then using (\ref{xa.4}), (\ref{xa.6}) to (\ref{xa.7}) we conclude that $\rho_{j}^{-k+1}\frac{(f^n)^{(1)}(\omega_{j}+z_{j}+\rho_{j}\zeta)}{\alpha(\omega_{j}+z_{j}+\rho_{j}\zeta)}\rightarrow h^{(1)}(\zeta)$ spherically locally uniformly in $\mathbb{C}$.
Therefore by mathematical induction, we can prove that
\bea\label{xa.8} h_{j}^{(i)}(\zeta)=\frac{(f^n)^{(i)}(\omega_{j}+z_{j}+\rho_{j}\zeta)}{\alpha(\omega_{j}+z_{j}+\rho_{j}\zeta)}\rightarrow h^{(i)}(\zeta)\eea
spherically locally uniformly in $\mathbb{C}$, where $i$ is any positive integer.

Now from (\ref{xa.4}) and (\ref{xa.8}) we see that
\beas \frac{h_{j}^{(1)}(0)}{h_{j}(0)}=\rho_{j}\frac{(f^n)^{(1)}(w_{j}+z_{j})}{f^n(w_{j}+z_{j})}\rightarrow \frac{h^{(1)}(0)}{h(0)}\eeas
and so from (\ref{1.1}), we get
\bea\label{xa.15} \rho_{j}\left|\frac{(f^n)^{(1)}(\omega_{j}+z_{j})}{f^n(\omega_{j}+z_{j})}\right|=\frac{1+|f^n(\omega_{j}+z_{j})|^{2}}{|(f^n)^{(1)}(\omega_{j}+z_{j})|}\frac{|(f^n)^{(1)}(\omega_{j}+z_{j})|}{|f^n(\omega_{j}+z_{j})|}=\frac{1+|f^n(\omega_{j}+z_{j})|^{2}}{|f^n(\omega_{j}+z_{j})|}\rightarrow \left|\frac{h^{(1)}(0)}{h(0)}\right|.\eea

Let
\beas\label{xa.9} (\hat{h}_{j})^{(k)}(\zeta)=\frac{(g^n)^{(k)}(\omega_{j}+z_{j}+\rho_{j}\zeta)}{\alpha(\omega_{j}+z_{j}+\rho_{j}\zeta)}.\eeas

Then from (\ref{al.1}), we have
\bea\label{xa.10}(h_{j})^{(k)}(\zeta)(\hat{h}_{j})^{(k)}(\zeta)=1.\eea

Clearly applying (\ref{xa.8}) to (\ref{xa.10}) we conclude that $(\hat{h}_{j})^{(k)}(\zeta)\rightarrow \hat{h}_{1}(\zeta)$
spherically locally uniformly in $\mathbb{C}$, where $\hat{h}_{1}$ is a non-constant entire function having no zeros.
Applying the formula of higher derivatives, we can deduce that $\hat{h}_{j}(\zeta)\rightarrow \hat{h}(\zeta)$
spherically locally uniformly in $\mathbb{C}$.
Therefore from (\ref{xa.8}) and (\ref{xa.10}) we get
\bea\label{xa.13} h^{(k)}(\zeta)\hat{h}^{(k)}(\zeta)\equiv 1.\eea

Since $\rho(h)\leq 1$, from (\ref{xa.13}) we get $\rho(\hat{h})\leq 1$. As $h$ and $\hat{h}$ are non-constant entire functions having no zeros, it follows that $\rho(h)=1$ and $\rho(\hat{h})=1$.
In this case from (\ref{xa.13}), we can conclude that $h(z)=c_{1}\exp(cz)$ and $\hat{h}(z)=c_{2}\exp(-cz)$,
where $c$, $c_{1}$ and $c_{2}$ are non-zero constants.
Now from (\ref{xa.15}), we get
\beas \frac{1+|f^n(\omega_{j}+z_{j})|^{2}}{|f^n(\omega_{j}+z_{j})|}\rightarrow |c|,\eeas
which shows that $\lim\limits_{j\rightarrow \infty} f^n(\omega_{j}+z_{j})\not=0,\infty$ and so
 (\ref{xa.4}), we have $h_{j}(0)=\rho_{j}^{-k}\frac{f^n(\omega_{j}+z_{j})}{ \alpha(\omega_{j}+z_{j})}\rightarrow \infty$.
On the other hand from (\ref{xa.4}), we get $h_{j}(0)\rightarrow h(0)=c_{1}$.
Consequently, we get a contradiction.\par

This completes the proof.

\end{proof}

\begin{proof}[{\bf Proof of Theorem \ref{t2.2}}]
Suppose
\be\label{al.2} (f^{n})^{(k)}(g^{n})^{(k)}\equiv \alpha^{2}.\ee

Clearly $\rho(f)=\rho(g)$. Since $T(r,f)=O(T(r,(f^n)^{(k)}))$ and $T(r,(f^n)^{(k)})=O(T(r,f))$, we can conclude that $\rho_2(f)=\rho_2((f^n)^{(k)})$. Similarly $\rho_2(g)=\rho_2((g^n)^{(k)})$. Again from (\ref{al.2}) we can deduce that $\rho_2((f^n)^{(k)})=\rho_2((g^n)^{(k)})$. Therefore we conclude that
\bea\label{aaa.12} \rho_2(f)=\rho_2(g).\eea

Let $f=\frac{f_{11}}{f_{12}}$, where $f_{11}$ and $f_{12}\not\equiv 0$ are entire functions.
Then by Weierstrass factorization theorem, we can write
$f_{11}(z)=z^{k_{11}}P_{1,0}(z)\exp(\delta_{11}(z))$, where $k_{11}$ is of multiplicity of $0$, $P_{1,0}$ is the product of primary factors formed with the non-null zeros of $f$ and $\delta_{11}$ is an entire function. Again by Weierstrass factorization theorem, we can write
$f_{12}(z)=z^{k_{12}}P_{1,\infty}(z)\exp(\delta_{12}(z))$,
where $k_{12}$ is of multiplicity of $0$, $P_{1,\infty}$ is the product of primary factors formed with the non-null poles of $f$ and $\delta_{12}$ is an entire function. Therefore
\bea\label{a1.3} f=R_1\exp(\delta_1),\eea
where $R_1(z)=z^{k_{1}}\frac{P_{1,0}(z)}{P_{1,\infty}(z)}$, $k_1=k_{11}-k_{12}$ and $\delta_1=\delta_{11}-\delta_{12}$. Similarly we can take
\bea\label{a1.4} g=R_2\exp(\delta_2),\eea
where $R_2(z)=z^{k_{2}}\frac{P_{2,0}(z)}{P_{2,\infty}(z)}$, $k_1=k_{21}-k_{22}$ and $\delta_2=\delta_{21}-\delta_{22}$.
It is to be noted that both $R_1$ and $R_2$ are non-zeros rational functions only when both $f$ and $g$ have finitely many zeros and poles.

First we want to estimate the counting functions $\ol N(r,0;f)$ and $\ol N(r,0;g)$.

For this let $z_{0}$ be a zero of $f$ of multiplicity $p_0$. Then $z_{0}$ is also a zero of $(f^{n})^{(k)}$ of multiplicity $np_0-k$.
Therefore one of the following possibilities holds for $z_0$:
\begin{enumerate}
\item[(i)] $z_{0}$ is neither a zero of $(g^{n})^{(k)}$ nor a pole of $g$;
\item[(ii)] $z_{0}$ is a zero of $(g^{n})^{(k)}$ but not a zero of $g$;
\item[(iii)] $z_{0}$ is a zero of $g$;
\item[(iv)] $z_{0}$ is a pole of $g$.
\end{enumerate}

For the possibility $(i)$: In this case, from (\ref{al.2}) we see that $z_{0}$ is a zero of $\alpha$ of multiplicity $\frac{np_0-k}{2}\geq p_0$.

For the possibility $(ii)$: Let $z_0$ be a zero of $(g^n)^{(k)}$ of multiplicity $q_0$ such that $g(z_0)\neq 0$. In this case, from (\ref{al.2}), we see that $z_{0}$ is a zero of $\alpha$ of multiplicity $\frac{np_0-k+q_0}{2}>p_0$.

For the possibility $(iii)$: Let $z_0$ be a zero of $g$ of multiplicity $q_0$. Then from (\ref{al.2}), we see that $z_{0}$ is a zero of $\alpha$ of multiplicity $\frac{n(p_0+q_0)-2k}{2}\geq p_0$.

For the possibility $(iv)$: Let $z_0$ be a pole of $g$ of multiplicity $s_0$. Since $n>2k$, it is not possible that $np_0-k=ns_0+k$.
 If $np_0-k>ns_0+k$, then from (\ref{al.2}) we see that $z_{0}$ is a zero of $\alpha$ of multiplicity $\frac{n(p_0-s_0)-2k}{2}\geq 1$. On the other hand if $np_0-k<ns_0+k$, then from (\ref{al.2}) we see that $z_{0}$ is a pole of $\alpha$ of multiplicity $\frac{n(s_0-p_0)+2k}{2}\geq 1$. Since $n>2k$, we must have $p_0\leq s_0$.

Therefore we conclude that
\bea\label{aaa.1}\label{aaa.5} \ol N(r,0;f)\leq N(r,0;\alpha)+N(r,\alpha).\eea

Note that if $\infty$ is a Picard exceptional value of $g$, then we have $N(r,0;f)\leq N(r,0;\alpha)$.

Similarly we have
\bea\label{aaa.6} \ol N(r,0;g)\leq N(r,0;\alpha)+N(r,\alpha).\eea

If $\infty$ is a Picard exceptional value of $f$, then we have
\bea\label{aaa.2} N(r,0;g)\leq N(r,0;\alpha).\eea

Next we want to estimate the counting functions $\ol N(r,f)$ and $\ol N(r,g)$.

For this let $z_{1}$ be a pole of $f$ of multiplicity $p_1$. Then $z_{1}$ is also a pole of $(f^{n})^{(k)}$ of multiplicity $np_1+k$.
Then one of the possibilities holds for $z_1$: $(I)$ $z_{1}$ is a zero of $(g^{n})^{(k)}$ but not a zero of $g$; $(II)$ $z_{1}$ is neither a zero of $(g^{n})^{(k)}$ nor a pole of $g$; $(III)$ $z_{1}$ is a zero of $g$ and $(IV)$ $z_{1}$ is a pole of $g$.

For the possibility $(I)$: Let $z_1$ be a zero of $(g^n)^{(k)}$ of multiplicity $q_1$ such that $g(z_1)\neq 0$. If $q_1>np_1+k$, then $z_1$ must be a zero of $\alpha$. If $q_1<np_1+k$, then $z_1$ is a pole of $\alpha$. But nothing can be said when $q_1=np_1+k$ and so we can not estimate the counting function
\bea\label{aa.2a} N_E(r,\infty;f^n\mid (g^n)^{(k)}=0\mid g^n\not=0)\eea
with the help of zeros or poles of $\alpha$.

For the remaining possibilities, $z_1$ must be either a zero of $\alpha$ or a pole of $\alpha$.
Therefore
\bea\label{aaa.3}\label{r.2} \ol N(r,f)\leq \ol N_E(r,\infty;f^n\mid (g^n)^{(k)}=0\mid g^n\not=0)+N(r,0;\alpha)+N(r,\alpha).\eea
Similarly we have
\bea\label{aaa.4}\label{r.3} \ol N(r,g)\leq \ol N_E(r,\infty;g^n\mid (f^n)^{(k)}=0\mid f^n\not=0)+N(r,0;\alpha)+N(r,\alpha).\eea

Also from the discussion, we see that if $z_0$ is a zero of $f$ of multiplicity $p_0$ and a pole of $g$ of multiplicity $s_0$, then $z_{0}$ will be a pole of $\alpha$ of multiplicity $\frac{n(s_0-p_0)+2k}{2}$, where $p_0\leq s_0$ and $np_0-k<ns_0+k$. Certainly $z_0$ will be a pole of $R_1R_2$ of multiplicity $s_0-p_0\leq \frac{n(s_0-p_0)+2k}{2}$. Consequently we conclude that
\bea\label{aaa.10} &&N(r,0;R_1R_2)+N(r,R_1R_2)\\&\leq& 2N(r,0;\alpha)+2 N(r,\infty;\alpha)+N_E(r,\infty;f^n\mid (g^n)^{(k)}=0\mid g^n\not=0)\nonumber\\&&+N_E(r,\infty;g^n\mid (f^n)^{(k)}=0\mid f^n\not=0)\nonumber\\&\leq&
2N(r,0;\alpha)+3N(r,\infty;\alpha)+N(r,0;(g^n)^{(k)}=0\mid g^n\not=0)+N(r,0;(f^n)^{(k)}=0\mid f^n\not=0)
.\nonumber\eea

Now using Lemma \ref{l2.6} to (\ref{al.2}), we conclude from (\ref{aaa.5}) and (\ref{aaa.6}) that
\bea\label{aaa.7} (k+1)\;\ol N(r,f)&=& N_{k+1}(r,(f^{n})^{(k)})\\&=&N_{k+1}\left(r,\frac{\alpha^{2}}{(g^{n})^{(k)}}\right)\nonumber\\&\leq& N_{k+1}(r,0;(g^{n})^{(k)})+N_{k+1}(r,\alpha)\nonumber\\
&\leq& N_{2k+1}(r,0;g^{n})+k\;\ol N(r,g^{n})+(k+1)N(r,\alpha)+O(\log T(r,g)+\log r)\nonumber\\
&\leq& (2k+1)\ol N(r,0;g)+k\;\ol N(r,g)+(k+1)N(r,\alpha)+O(\log T(r,g)+\log r)\nonumber\\&\leq& k\ol N(r,g)+(3k+2)\left[N(r,0;\alpha)+N(r,\alpha)\right]+O(\log T(r,g)+\log r)\nonumber.\eea
Similarly
\bea\label{aaa.8}(k+1)\;\ol N(r,g)\leq k\;\ol N(r,f)+(3k+2)\left[N(r,0;\alpha)+N(r,\alpha)\right]+O(\log T(r,f)+\log r).\eea
Therefore from (\ref{aaa.7}) and (\ref{aaa.8}), we have
\bea\label{aaa.9}\ol N(r,f)+\ol N(r,g)&\leq& (6k+4)\left[N(r,0;\alpha)+N(r,\alpha)\right]+O(\log T(r,f)+\log r)\\&&+O(\log T(r,g)+\log r)\nonumber.\eea

Suppose $F=\frac{f^n}{\alpha}$ and $G=\frac{g^n}{\alpha}$.
Let $\mathcal{F}=\{F_{\omega}\}$ and $\mathcal{G}=\{G_{\omega}\}$, where $F_{\omega}(z)=F(z+\omega)$ and $G_{\omega}(z)=G(z+\omega)$, $z\in \mathbb{C}$. Clearly $\mathcal{F}$ and $\mathcal{G}$ are two families of meromorphic functions defined on $\mathbb{C}$.\par

We now divide following two cases.\par
{\bf Case 1.} Let $\rho(\alpha)=0$. In this case $\alpha$ has only finitely many zeros and poles and so $\alpha$ is a rational function. Then from (\ref{aaa.5}) and (\ref{aaa.6}), we conclude that both $f$ and $g$ have finitely many zeros.

We now consider the following two sub-cases.\par
{\bf Sub-case 1.1.} Let $\mathcal{F}$ be normal on $\mathbb{C}$. Then by Marty's theorem $F^{\#}(\omega)=F^{\#}_{\omega}(0)\leq M$ for some $M>0$ and for all $\omega\in\mathbb{C}$. Hence by Lemma \ref{l2.8}, we have $\rho(F)\leq 2$. Since $\rho(\alpha)=0$, it follows that $\rho(f)\leq 2$ and so $\rho(g)\leq 2$. Clearly from (\ref{a1.3}) and (\ref{a1.4}) we deduce that $\deg(\delta_1)\leq 2$ and $\deg(\delta_2)\leq 2$.
On the other hand from (\ref{aaa.9}), we get $\ol N(r,f)+\ol N(r,g)=O(\log r)$ as $r\to \infty$. This shows that $f$ and $g$ have at most finitely many poles. Therefore from (\ref{a1.3}) and (\ref{a1.4}) we conclude that both $R_1$ and $R_2$ are rational functions.
Then from (\ref{al.1.5}) we see that $\delta_1+\delta_2$ is a constant and so $\delta_1^{(1)}=-\delta_2^{(1)}$.
Now from (\ref{xr.3}) we get $2k\deg(\delta_1^{(1)})=\deg(\alpha^2/(R_1R_2)^n)$. In particular if $\alpha\equiv 1$, then both $R_1$ and $R_2$ can not have zeros.
Since $n>2k$, from $2k\deg(\delta_1^{(1)})=\deg(\alpha^2/(R_1R_2)^n)$ we conclude that $\deg(\delta_1)=1$ and $R_{1}$, $R_{2}$ reduce to non-zero constants. Finally $f=R_1\exp(\delta_1)$ and $g=R_2\exp(\delta_2)$, where $R_{1}$, $R_{2}$ are non-zero rational functions and $\delta_1$, $\delta_2$ are non-constant polynomials of degree at most two such that $\delta_1^{(1)}=-\delta_2^{(1)}$ and $2k\deg(\delta_1^{(1)})=2\deg(\alpha)-n\deg(R_1R_2)$. In particular if $\alpha\equiv 1$, then $f(z)=c_1\exp(cz)$ and $g(z)=c_2\exp(-cz)$, where $c$, $c_1$ and $c_2$ are non-zero constants such that $(-1)^k(c_1c_2)^n(nc)^{2k}=1$.\par

{\bf Sub-case 1.2.} Let $\mathcal{F}$ be not normal on $\mathbb{C}$. Suppose $\mathcal{F}$ is not normal at $z_{0}=0$.
Then proceeding in the same way as done in Sub-case 2.2 in the proof of Theorem \ref{t2.1}, we can easily construct two families
$\{F(z+\omega_{j})\}$ and $\{G(z+\omega_{j})\}$ of meromorphic functions in $z\in\Delta$ having no zeros.
Now if we consider the counting function given by (\ref{aa.2a}), then we conclude that the multiplicity of every pole of $F(z+\omega_{j})$ is a multiple of $n$. Same as for $G(z+\omega_{j})$.
Clearly we apply Lemma 1.A and we get
\bea\label{a.4} h_{j}(\zeta)=\rho_j^{-k}F_{j}(z_{j}+\rho_{j}\zeta)=\frac{f^n(\omega_{j}+z_{j}+\rho_{j}\zeta)}{\alpha(\omega_{j}+z_{j}+\rho_{j}\zeta)}\rightarrow h(\zeta)\eea
spherically locally uniformly in $\mathbb{C}$, where $h$ is a non-constant meromorphic function having no zeros such that $h^{\#}(\zeta)\leq h^{\#}(0)=1$. Clearly from Lemma \ref{l2.8}, we get $\rho(h)\leq 2$.
By Hurwitz's theorem we can say that the multiplicity of every pole of $h$ is a multiple of $n$. So we assume that $h=\bar{h}^{n}$, where $\bar{h}$ is some non-constant meromorphic function.
If we take $(\hat{h}_{j}(\zeta))^{(k)}=\frac{(g^n)^{(k)}(\omega_{j}+z_{j}+\rho_{j}\zeta)}{\alpha(\omega_{j}+z_{j}+\rho_{j}\zeta)}$,
then proceeding in the same way as done in Sub-case 2.2 in the proof of Theorem \ref{t2.1}, we can prove that
$\hat{h}_{j}(\zeta)\rightarrow \hat{h}_{1}(\zeta)$
spherically locally uniformly in $\mathbb{C}$, where $\hat{h}_{1}$ is some non-constant meromorphic function having no zeros. Using Hurwitz's theorem, we conclude that the multiplicity of every pole of $\hat{h}_{1}$ is a multiple of $n$ and so we can assume that $\hat{h}_{1}=\hat{h}^{n}$, where $\hat{h}$ is some non-constant meromorphic function.
Since $(h_{j}(\zeta))^{(k)}\rightarrow (\bar{h}^{n}(\zeta))^{(k)}$ and $(\hat{h}_{j}(\zeta))^{(k)}\rightarrow (\hat{h}^{n}(\zeta))^{(k)}$, from (\ref{al.2}) we get
\bea\label{a.13} (\bar{h}^{n}(\zeta))^{(k)}(\hat{h}^{n}(\zeta))^{(k)}\equiv 1.\eea

Since $\rho(h)\leq 2$, from (\ref{a.13}) we get $\rho(\hat{h})\leq 2$. Again from (\ref{aaa.9}) one can easily deduce that
$\ol N(r,\bar{h})+\ol N(r,\hat{h})=O(\log r)$. This shows that $\bar{h}$ and $\hat{h}$ have finitely many poles. Let $\bar{h}=\frac{1}{Q_{3}}\exp(\delta_{11})$ and $\hat{h}=\frac{1}{Q_{4}}\exp(\delta_{12})$,
where $Q_{3}$, $Q_{4}$ are non-zero polynomials and $\delta_{11}$, $\delta_{12}$ are non-constant polynomials of degree at most two. Now proceeding in the same manner as in Sub-case 1.1, we can conclude that $Q_{3}$, $Q_{4}$ are non-zero constants and $\delta_{11}$, $\delta_{12}$ are polynomials of degree one such that $\delta_{11}^{(1)}=-\delta_{12}^{(1)}$. So from (\ref{a.13}), we get $\bar{h}(z)=\bar{c}_{1}\exp(cz)$ and $\hat{h}(z)=\hat{c}_{2}\exp(-cz)$, where $c$, $\bar{c}_{1}$ and $\hat{c}_{2}$ are non-zero constants.

Now proceeding in the same manner as done in the proof of Sub-case 2.2 of Theorem \ref{t2.1}, we get a contradiction.\par

{\bf Case 2.} Let $\rho(\alpha)>0$. In this case, we have either $N(r,0;\alpha)+N(r,\alpha)=O(\log r)$ or $N(r,0;\alpha)+N(r,\alpha)\neq O(\log r)$. So
we have to consider following two sub-cases.\par
{\bf Sub-case 2.1.} Let $N(r,0;\alpha)+N(r,\alpha)=O(\log r)$ as $r\to\infty$. Consequently  $\alpha=R\exp(P)$, where $R$ is a non-zero rational function and $P$ is a polynomial. Also from (\ref{aaa.5}) and (\ref{aaa.6}), we conclude that both $f$ and $g$ have finitely many zeros.

Now we divide following two sub-cases.\par

{\bf Sub-case 2.1.1.} Let $\rho(f)<+\infty$. Clearly $\rho(g)<+\infty$. Now from (\ref{aaa.9}), we get $\ol N(r,f)+\ol N(r,g)=O(\log r)$ as $r\to \infty$. This shows that $f$ and $g$ have at most finitely many poles. Therefore from (\ref{a1.3}) and (\ref{a1.4}) we see that both $R_1$ and $R_2$ are rational functions and
$\delta_1$, $\delta_2$ are polynomials such that $\deg(\delta_1)=\rho(f)$ and $\deg(\delta_2)=\rho(g)$.
Then from (\ref{al.1.5})-(\ref{al.1.5bb}), we get
\bea\label{yr.3}&&\left(\left(\frac{(R_1^n)^{(1)}}{R_1^n}+n\delta_1^{(1)}\right)^{k}+P^{*}_{k-1}\left(\frac{(R_1^n)^{(1)}}{R_1^n}+n\delta_1^{(1)}\right)\right)\left(\left(\frac{(R_2^n)^{(1)}}{R_2^n}+n\delta_2^{(1)}\right)^{k}+P^{*}_{k-1}\left(\frac{(R_2^n)^{(1)}}{R_2^n}+n\delta_2^{(1)}\right)\right)\nonumber\\&&=\frac{\exp(-n(\delta_1+\delta_2)+2P)R^2}{(R_1R_2)^{n}}.\eea

Clearly from (\ref{yr.3}) we conclude that $n(\delta_1+\delta_2)-P$ is a constant. Since $\rho(f)=\rho(g)$, we have $\deg(\delta_1)=\deg(\delta_2)$ and so $\delta_1^{(1)}\not\equiv 0$. Therefore letting $|z|\to \infty$ on the both sides of (\ref{yr.3}) we get $2k\deg(\delta_1^{(1)})=2\deg(R)-n\deg(R_1R_2)$.
Finally $f=R_1\exp(\delta_1)$ and $g=R_2\exp(\delta_2)$, where $R_{1}$, $R_{2}$ are non-zero rational functions and $\delta_1$, $\delta_2$ are non-constant polynomials such that $\deg(\delta_1)=\deg(\delta_2)$, $n(\delta_1+\delta_2)-P$ is a constant and $2k\deg(\delta_1^{(1)})=2\deg(R)-n\deg(R_1R_2)$.\par

{\bf Sub-case 2.1.2.} Let $\rho(f)=+\infty$. Since $\rho(\alpha)<\rho(f)$, it follows that $\rho(F)=+\infty$.
Then by Lemma \ref{l2.7}, there exist a sequence $\{\omega_{j}\}$ with $\omega_{j}\rightarrow \infty$ such that for every $N>0$
\be\label{als.1} F^{\#}(\omega_{j})>|\omega_{j}|^{N},\ee if $j$ is sufficiently large.
Note that $F_{j}^{\#}(0)=F^{\#}(\omega_{j})\rightarrow \infty$ as $j\rightarrow \infty$. Then by Theorem 1.D, we conclude that $\{F_{j}\}$ is not normal at $z=0$.
We use the same methodology that is similar to Sub-case 2.2 of Theorem \ref{t2.1}, but with necessary modifications as follow:\par
From (\ref{1.1}) we have $\rho_{j}=\frac{1}{F^{\#}(\omega_{j}+z_j)}$.
Also we know that $F^{\#}(\omega_j+z_j)=F_j^{\#}(z_j)\geq F_j^{\#}(0)=F^{\#}(\omega_j)$. Therefore $\rho_j\leq \frac{1}{F^{\#}(\omega_j)}$ and so
from (\ref{als.1}), we see that for every $N>0$
\bea\label{als.3} \rho_{j}<|\omega_{j}|^{-N},\eea
if $j$ is sufficiently large. Also we have
\bea\label{als.4} \frac{\alpha^{(1)}(\omega_{j}+z_{j}+\rho_{j}\zeta)}{\alpha(\omega_{j}+z_{j}+\rho_{j}\zeta)}=\frac{R^{(1)}(\omega_{j}+z_{j}+\rho_{j}\zeta)}{R(\omega_{j}+z_{j}+\rho_{j}\zeta)}+P^{(1)}(\omega_{j}+z_{j}+\rho_{j}\zeta).\eea

Observe that $\frac{R^{(1)}(\omega_{j}+z_{j}+\rho_{j}\zeta)}{R(\omega_{j}+z_{j}+\rho_{j}\zeta)}\rightarrow 0$ as $j\rightarrow \infty$.
Let $\deg(P^{(1)})=s$. Suppose $N>s$. Then from (\ref{als.3}) we get $\lim\limits_{j\rightarrow\infty}\rho_{j}|\omega_{j}|^{s}\leq \lim\limits_{j\rightarrow \infty} |\omega_{j}|^{s-N}=0$. Since $|P^{(1)}(\omega_{j}+z_{j}+\rho_{j}\zeta)|=O(|\omega_{j}|^{s})$, we have
\bea\label{als.6} \rho_{j} |P^{(1)}(\omega_{j}+z_{j}+\rho_{j}\zeta)|=O(\rho_{j} |\omega_{j}|^{s})\rightarrow 0\;\;\text{as}\;\;j\rightarrow \infty.\eea

Then from (\ref{als.4}) and (\ref{als.6}), we have
\beas \rho_{j} \frac{\alpha^{(1)}(\omega_{j}+z_{j}+\rho_{j}\zeta)}{\alpha(\omega_{j}+z_{j}+\rho_{j}\zeta)}\rightarrow 0\;\;\text{as}\;\;j\rightarrow \infty.\eeas

Now it is clear to us that we must get a contradiction if we follow the proof of Sub-case 2.2 of Theorem \ref{t2.1}.\par

{\bf Sub-case 2.2.} Let $N(r,0;\alpha)+N(r,\alpha)\not=O(\log r)$ as $r\to\infty$. Then for a given $\varepsilon>0$, there exists $R>0$ such that
$N(r,\alpha)<r^{\rho_1(\frac{1}{\alpha})+\varepsilon}$ and $N\left(r,\frac{1}{\alpha}\right)<r^{\rho_1(\alpha)+\varepsilon}$, $\forall\;r>R$.
Let $\rho_1=\max\{\rho_1(\frac{1}{\alpha}),\rho_1(\alpha)\}$.
Therefore
$N(r,\alpha)<r^{\rho_1+\varepsilon}$ and $N\left(r,\frac{1}{\alpha}\right)<r^{\rho_1+\varepsilon}$, $\forall\;r>R$. Consequently from (\ref{aaa.5}) and (\ref{aaa.6}), we deduce that
\bea\label{xa.1}\ol N(r,0;f)<r^{\rho_1+\varepsilon}\;\;\text{and}\;\;\ol N(r,0;g)<r^{\rho_1+\varepsilon}\eea
 $\forall\;r>R$.
Following two sub-cases are immediately.\par
{\bf Sub-case 2.2.1.} Let $\rho(f)<+\infty$. Then from (\ref{aaa.9}), we have
\[\ol N(r,f)<r^{\rho_1+\varepsilon}\;\;\text{and}\;\;\ol N(r,g)<r^{\rho_1+\varepsilon}\]
$\forall\;r>R$. Using Theorem 1.A, we get
\beas T(r,\tilde P)=T\Big(r,\frac{(f^n)^{(k)}}{f^n}\Big)\leq& \ol N(r,f)+(k+1)\ol N(r,0;f)+O(\log r)<r^{\rho_1+\varepsilon},\eeas
$\forall\;r>R$. Similarly
\[T(r,\tilde Q)=T\Big(r,\frac{(g^n)^{(k)}}{g^n}\Big)<r^{\rho_1+\varepsilon},\]
$\forall\;r>R$. Consequently $\rho(\tilde P)\leq \rho_1<\rho(\alpha)=\mu(\alpha)$ and $\rho(\tilde Q)\leq \rho_1<\rho(\alpha)=\mu(\alpha)$. Then by Theorem 1.B, we have $T(r,\tilde P)=S(r,\alpha)$ and $T(r,\tilde Q)=S(r,\alpha)$.

Again since $N(r,0;(f^n)^{(k)}\mid f\neq 0)\leq T(r,\tilde P)$ and $N(r,0;(g^n)^{(k)}\mid g\neq 0)\leq T(r,\tilde Q)$, we get
\[N(r,0;(f^n)^{(k)}\mid f\neq 0)<r^{\rho_1+\varepsilon}\;\;\text{and}\;\;N(r,0;(g^n)^{(k)}\mid g\neq 0)<r^{\rho_1+\varepsilon},\]
$\forall\;r>R$. Consequently from (\ref{aaa.10}), we have
\[\rho_1(0;R_1R_2)\leq \rho_1<\rho(\alpha)\;\;\text{and}\;\;\rho_1(\infty;R_1R_2)\leq \rho_1<\rho(\alpha).\]

After cancellation of the common factors, we may assume that
\[R_1R_2=z^{k_1+k_2}\frac{P_{1,0}P_{2,0}}{P_{1,\infty}P_{2,\infty}}=z^{k_3}\frac{P_{3,0}}{P_{3,\infty}},\]
where $k_3=k_1+k_2$, $P_{3,0}$ and $P_{3,\infty}$ are the product of primary factors formed with the non-null zeros and poles respectively.
Now by Ash \cite[Theorem 4.3.6]{1}, we get $\rho(P_{3,0})=\rho_1(0;P_{3,0})\leq \rho_1<\rho(\alpha)$ and $\rho(P_{3,\infty})=\rho_1(0;P_{3,\infty})\leq \rho_1<\rho(\alpha)$. Clearly $\rho(R_1R_2)<\rho(\alpha)=\mu(\alpha)$. Then by Theorem 1.B, we have $T(r,R_1R_2)=S(r,\alpha)$.

Since $T(r,\tilde P)=S(r,\alpha)$, $T(r,\tilde Q)=S(r,\alpha)$ and $T(r,R_1R_2)=S(r,\alpha)$, from (\ref{al.1.5}) we conclude that $\deg(\delta_1+\delta_2)=\rho(\alpha)$. Also from above we have $\ol{\rho}_1(0;f)<\rho(\alpha)$, $\ol{\rho}_1(\infty;f)<\rho(\alpha)$, $\ol {\rho}_1(0;g)<\rho(\alpha)$ and $\ol{\rho}_1(\infty;g)<\rho(\alpha)$.
Therefore $f=R_1\exp(\delta_1)$ and $g=R_2\exp(\delta_2)$, where $\ol{\rho}_1(0;R_i)<\rho(\alpha)$, $\ol{\rho}_1(\infty;R_i)<\rho(\alpha)$ for $i=1,2$, $R_{1}R_{2}$ is a small function of $\alpha$, $\delta_1$ and $\delta_1$ are non-constant polynomials such that $\deg(\delta_1+\delta_2)=\rho(\alpha)$.\par

{\bf Sub-case 2.2.2.} Let $\rho(f)=+\infty$. Since $\rho(f)=\rho(g)$, we have $\rho(f)=+\infty$.  In this case, we have $f=R_1\exp(\delta_1)$ and $g=R_2\exp(\delta_2)$, where $\delta_1$ and $\delta_2$ are transcendental entire functions. Next we consider following two sub-cases.\par

{\bf Sub-case 2.2.2.1.} Let $\rho_2(f)<+\infty$. Then from (\ref{aaa.12}), we have $\rho_2(f)=\rho_2(g)<\rho(f)$. Consequently we have
$\log T(r,f)<r^{\rho_2(f)+\varepsilon}$ and $\log T(r,g)<r^{\rho_2(f)+\varepsilon}$, $\forall$ $r>R$. Let $\rho_2=\max\{\rho_1(\alpha), \rho_1(\frac{1}{\alpha}), \rho_2(f)\}$. Now from (\ref{aaa.9}), we get $\ol N(r,f)=O(r^{\rho_2+\varepsilon})$ and $\ol N(r,g)=O(r^{\rho_2+\varepsilon})$, $\forall$ $r>R$. Consequently
$\ol{\rho}_1(\infty;f)\leq\rho_2<\rho(f)$ and $\ol{\rho}_1(\infty;g)\leq \rho_2<\rho(g)$. Again from (\ref{xa.1}) we get
$\ol{\rho}_1(0;f)\leq\rho_2<\rho(f)$ and $\ol {\rho}_1(0;g)\leq\rho_2<\rho(g)$. Using Theorem 1.A, we get
\[T\left(r,\frac{(f^n)^{(1)}}{f^n}\right)=\ol N(r,f)+\ol N(r,0;f)+O(\log T(r,f)+\log r)=O(r^{\rho_2+\varepsilon})\]
$\forall$ $r>R$ and so
$\rho\left(\frac{(f^n)^{(1)}}{f^n}\right)\leq \rho_2<\rho(f)$. Similarly we can prove that $\rho\left(\frac{(g^n)^{(1)}}{g^n}\right)\leq \rho_2<\rho(f)$.
Consequently from (\ref{alr.20}), we have $\rho(\tilde P)\leq \rho_2<\rho(f)$ and $\rho(\tilde Q)\leq \rho_2<\rho(f)$.

Again since $N(r,0;(f^n)^{(k)}\mid f\neq 0)\leq T(r,\tilde P)$ and $N(r,0;(g^n)^{(k)}\mid g\neq 0)\leq T(r,\tilde Q)$, we have
\[N(r,0;(f^n)^{(k)}\mid f\neq 0)=O(r^{\rho_2+\varepsilon})\;\;\text{and}\;\;N(r,0;(g^n)^{(k)}\mid g\neq 0)=O(r^{\rho_2+\varepsilon}),\]
$\forall\;r>R$. Consequently from (\ref{aaa.10}), we deduce that
\[\rho_1(0;R_1R_2)\leq \rho_2<\rho(f)\;\;\text{and}\;\;\rho_1(\infty;R_1R_2)\leq \rho_2<\rho(f).\]

If we may assume that
\[R_1R_2=z^{k_1+k_2}\frac{P_{1,0}P_{2,0}}{P_{1,\infty}P_{2,\infty}}=z^{k_3}\frac{P_{3,0}}{P_{3,\infty}},\]
where $k_3=k_1+k_2$, $P_{3,0}$ and $P_{3,\infty}$ are the product of primary factors formed with the non-null zeros and poles respectively, then by Ash \cite[Theorem 4.3.6]{1} we get
$\rho(P_{3,0})=\rho_1(0;P_{3,0})\leq \rho_2<\rho(f)$ and $\rho(P_{3,\infty})=\rho_1(0;P_{3,\infty})\leq \rho_2<\rho(f)$. Clearly $\rho(R_1R_2)\leq \rho_2<\rho(f)$. Consequently $\rho\left(\tilde P \tilde Q (R_1R_2)^n\right)\leq \rho_2<\rho(f)$. Since $\rho(\alpha)<+\infty$, it follows from (\ref{al.1.5}) that $\rho\left(e^{\delta_1+\delta_2}\right)\leq \rho_2<\rho(f)$ and so $\delta_1+\delta_2$ is a polynomial. Therefore $f=R_1\exp(\delta_1)$ and $g=R_2\exp(\delta_2)$, where $\ol{\rho}_1(0;R_i)<\rho(\alpha)$, $\ol{\rho}_1(\infty;R_i)<\rho_2(f)$ for $i=1,2$, and $\delta_1$, $\delta_2$ are transcendental entire functions such that $\delta_1+\delta_2$ is a polynomial.\par

{\bf Sub-case 2.2.2.2.} Let $\rho_2(f)=+\infty$. Then from (\ref{aaa.12}), we have $\rho_2(g)=+\infty$. In this case from (\ref{xa.1}), we deduce that
$\ol{\rho}_1(0;R_i)<\rho(\alpha)$ for $i=1,2$. On the other hand from (\ref{aaa.9}) we have $\ol N(r,R_i)=S(r,f)$ for $i=1,2$.
Therefore $f=R_1\exp(\delta_1)$ and $g=R_2\exp(\delta_2)$, where $\ol{\rho}_1(0;R_i)<\rho(\alpha)$, $\ol N(r,R_i)=S(r,f)$ for $i=1,2$, $\delta_1$ and $\delta_2$ are transcendental entire functions.

This completes the proof.

\end{proof}

\end{document}